\documentclass[12pt, letterpaper]{amsart}

\usepackage{tikz}
\usetikzlibrary{matrix,arrows,decorations.pathmorphing}
\usepackage{graphicx}
\usepackage{amssymb}
\usepackage{mathrsfs}
\usepackage{latexsym}
\usepackage{amsmath}
\usepackage[
hypertexnames=false, colorlinks, citecolor=red, linkcolor=red]{hyperref}
\hypersetup{bookmarksdepth=3}

	\normalbaselines						  %

\newcommand{\R}{{\mathbb  R}}
\newcommand{\kf}[2]{{\begin{pmatrix} #1 \\ #2 \end{pmatrix}}}

\newcommand{\D}{{\mathbb  D}}

\newcommand{\T}{\mathbb{T}}

\newcommand{\Z}{{\mathbb  Z}}
\newcommand{\N}{{\mathbb  N}}
\newcommand{\C}{{\mathbb  C}}

\newcommand{\OZ}{{\mathbf{0}}}
\newcommand{\dd}{{\mathrm{d}}}

\newcommand{\OID}{{\mathbf{I}}}
\newcommand{\bI}{\mathbf{I}}

\newcommand{\cP}{\mathcal{P}}

\newcommand{\cG}{\mathcal{G}}

\newcommand{\cC}{\mathcal{C}}

\newcommand{\I}{\mathbf{I}}

\newcommand{\0}{\mathbb{O}}
\newcommand{\bO}{\mathbf{0}}

\newcommand{\tto}{\!\!\to\!}

\newcommand{\bC}{\mathbf{C}}
\newcommand{\bB}{\mathbf{B}}
\newcommand{\be}{\mathbf{e}}

\newcommand{\btheta}{\boldsymbol{\theta}}

\newcommand{\wt}{\widetilde}

\newcommand{\rx}{{X}}
\newcommand{\er}{\mathfrak{R}}

\newcommand{\cM}{\mathcal{M}}
\newcommand{\cK}{{\mathcal K}}
\newcommand{\fD}{{\mathfrak D}}
\newcommand{\cH}{\mathcal{H}}
\newcommand{\cD}{\mathcal{D}}
\newcommand{\cX}{\mathcal{X}}
\newcommand{\cU}{\mathcal{U}}
\newcommand{\cE}{\mathcal{E}}
\newcommand{\cY}{\mathcal{Y}}
\newcommand{\cA}{\mathcal{A}}

\newcommand{\1}{\mathbf{1}}

\newcommand{\te}{\theta}
\newcommand{\f}{\varphi}

\newcommand{\e}{\varepsilon}

\DeclareMathOperator{\clos}{clos}

\DeclareMathOperator{\rank}{rank}

\DeclareMathOperator{\Ker}{Ker}
\DeclareMathOperator{\Ran}{Ran}
\DeclareMathOperator{\spa}{span}
\DeclareMathOperator{\cspa}{\overline{span}}

\DeclareMathOperator{\re}{Re}

\newcommand{\la}{\lambda}

\newcommand{\ci}[1]{_{_{\scriptstyle #1}}}

\newcommand{\ti}[1]{_{\scriptstyle \text{\rm #1}}}

\catcode`@=11
\chardef\mathlig@atcode\count255

\def\actively#1#2{\begingroup\uccode`\~=`#2\relax\uppercase{\endgroup#1~}}
\def\mathlig@gobble{\afterassignment\mathlig@next@cmd\let\mathlig@next= }
\def\mathlig@delim{\mathlig@delim}
\def\mathlig@defcs#1{\expandafter\def\csname#1\endcsname}
\def\mathlig@let@cs#1#2{\expandafter\let\expandafter#1\csname#2\endcsname}
\def\mathlig@appendcs#1#2{\expandafter\edef\csname#1\endcsname{\csname#1\endcsname#2}}

\def\mathlig#1#2{\mathlig@checklig#1\mathlig@end\mathlig@defcs{mathlig@back@#1}{#2}\ignorespaces}

\def\mathlig@checklig#1#2\mathlig@end{%
 \expandafter\ifx\csname mathlig@forw@#1\endcsname\relax
 \expandafter\mathchardef\csname mathlig@back@#1\endcsname=\mathcode`#1%
 \mathcode`#1"8000\actively\def#1{\csname mathlig@look@#1\endcsname}%
 \mathlig@dolig#1\mathlig@delim
\fi
\mathlig@checksuffix#1#2\mathlig@end
}

\def\mathlig@checksuffix#1#2\mathlig@end{%
\ifx\mathlig@delim#2\mathlig@delim\relax\else\mathlig@checksuffix@{#1}#2\mathlig@end\fi
}
\def\mathlig@checksuffix@#1#2#3\mathlig@end{%
\expandafter\ifx\csname mathlig@forw@#1#2\endcsname\relax\mathlig@dosuffix{#1}{#2}\fi
\mathlig@checksuffix{#1#2}#3\mathlig@end
}

\def\mathlig@dosuffix#1#2{%
\mathlig@appendcs{mathlig@toks@#1}{#2}%
\mathlig@dolig{#1}{#2}\mathlig@delim
}

\def\mathlig@dolig#1#2\mathlig@delim{%
 \mathlig@defcs{mathlig@look@#1#2}{%
 \mathlig@let@cs\mathlig@next{mathlig@forw@#1#2}\futurelet\mathlig@next@tok\mathlig@next}%
 \mathlig@defcs{mathlig@forw@#1#2}{%
  \mathlig@let@cs\mathlig@next{mathlig@back@#1#2}%
  \mathlig@let@cs\checker{mathlig@chck@#1#2}%
  \mathlig@let@cs\mathligtoks{mathlig@toks@#1#2}%
  \expandafter\ifx\expandafter\mathlig@delim\mathligtoks\mathlig@delim\relax\else
  \expandafter\checker\mathligtoks\mathlig@delim\fi
  \mathlig@next
 }%
 \mathlig@defcs{mathlig@toks@#1#2}{}%
 \mathlig@defcs{mathlig@chck@#1#2}##1##2\mathlig@delim{%
  \ifx\mathlig@next@tok##1%
   \mathlig@let@cs\mathlig@next@cmd{mathlig@look@#1#2##1}\let\mathlig@next\mathlig@gobble
  \fi
  \ifx\mathlig@delim##2\mathlig@delim\relax\else
   \csname mathlig@chck@#1#2\endcsname##2\mathlig@delim
  \fi
 }%
 \ifx\mathlig@delim#2\mathlig@delim\else
  \mathlig@defcs{mathlig@back@#1#2}{\csname mathlig@back@#1\endcsname #2}%
 \fi
}%

\catcode`@\mathlig@atcode

\mathchardef\ordinarycolon\mathcode`\:
\def\vcentcolon{\mathrel{\mathop\ordinarycolon}}
\mathlig{:=}{\vcentcolon=}
\mathlig{::=}{\vcentcolon\vcentcolon=}

\numberwithin{equation}{section}

\theoremstyle{plain}
\newtheorem{theo}{Theorem}[section]
\newtheorem{cor}[theo]{Corollary}
\newtheorem{lem}[theo]{Lemma}
\newtheorem{prop}[theo]{Proposition}

\theoremstyle{definition}
\newtheorem{defn}[theo]{Definition}

\theoremstyle{remark}
\newtheorem*{ex*}{Example}
\newtheorem*{exs*}{Examples}
\newtheorem{rem}[theo]{Remark}
\newtheorem*{rem*}{Remark}
\newtheorem*{rems*}{Remarks}

\newcounter{vremennyj}

\title[Finite rank perturbations]{General Clark model for finite rank perturbations}

\author{Constanze~Liaw}
\address{C.~Liaw: Department of Mathematical Sciences, University of Delaware, 311 Ewing Hall, Newark, DE 19716, USA and CASPER, Baylor University, One Bear Place \#97328,      
 Waco, TX  76798, USA}
\email{liaw@udel.edu}
\author{Sergei~Treil}
\address{S.~Treil: Department of Mathematics, Brown University   
151 Thayer
Str./Box 1917,      
 Providence, RI  02912, USA}
 
\email{treil@math.brown.edu}

 \thanks{
The work of C.~Liaw is supported by the National Science Foundation DMS-1802682. 
\newline
Work of S.~Treil is supported by the National Science Foundation under the grants DMS-1301579, DMS-1600139. Any opinions, findings and conclusions or recommendations expressed in this material are those of the author and do not necessarily reflect the views of the National Science Foundation. }

\begin{document}

\begin{abstract}
All unitary (contractive) perturbations of a given  unitary operator $U$ by finite rank $d$ operators with fixed range can be parametrized by $(d\times d)$ unitary (contractive) matrices $\Gamma$; this generalizes unitary rank one ($d=1$) perturbations, where the Aleksandrov--Clark family of unitary perturbations is parametrized by the scalars on the unit circle $\T\subset\C$.

For a strict contraction $\Gamma$  the resulting perturbed operator $T\ci \Gamma$ is (under the natural assumption about star cyclicity of the range) a completely non-unitary contraction, so it admits the functional model. 

In this paper we investigate the Clark operator, i.e.~a unitary operator that intertwines $T\ci\Gamma$ (written in the spectral representation of the non-perturbed operator $U$) and its model.  We make no assumptions on the spectral type of the  unitary operator $U$; absolutely continuous spectrum may be present.

We first find a universal representation of the adjoint Clark operator in the coordinate free Nikolski--Vasyunin functional model; the word ``universal'' means that it is valid in any transcription of the model.  This representation can be considered to be  a special version of the vector-valued Cauchy integral operator. 

Combining the theory of singular integral operators with the theory of functional models we derive from this abstract representation 
a concrete formula for the adjoint of the Clark operator  in the Sz.-Nagy--Foia\c s transcription. As in the scalar case the adjoint Clark operator is given by a sum of two terms: one is given by the boundary values of the vector-valued Cauchy transform (postmultiplied by a matrix-valued function) and the second one is just the multiplication operator by a matrix-valued function. 

Finally, we present formulas for the direct Clark operator in the Sz.-Nagy--Foia\c s transcription. 
\end{abstract}

\maketitle

\setcounter{tocdepth}{1}
\tableofcontents

\setcounter{section}{-1}

\section{Introduction}
The contractive (or unitary)  perturbations $U+K$ of a  unitary operator $U$ on a  Hilbert space $H$  by finite rank $d<\infty$ operators $K$ with fixed range are parametrized by the $(d\times d)$ contractive (resp.~unitary) matrices $\Gamma$. Namely, if $\Ran K\subset {\er}$, where ${\er}\subset H$, $\dim {\er}=d$ is fixed, and $\bB:\C^d\to {\er}$ is a fixed unitary operator (which we call the coordinate operator), then $K$ is represented as $K= \bB(\Gamma-\bI\ci{\C^d}) \bB^*U$ where $\Gamma$ is a contraction (resp.~a unitary operator) on $\C^d$. Therefore, all such perturbations with $\Ran K\subset {\er}$ are represented as  $T\ci{\Gamma}= U+ \bB(\Gamma-\bI\ci{\C^d}) \bB^*U$, where $\Gamma $ runs over all $(d\times d)$ contractive  (resp.~unitary)  matrices.  

Recall that $T$ being a \emph{contraction} (contractive) means that $\|T\|\le 1$. 


Focusing on the non-trivial part of the perturbation, we can assume that $\Ran \bB = {\er}$ is a star-cyclic subspace for $U$, i.e.~$H = \overline{\spa}\{U^k {\er}, (U^*)^k{\er}: k\in\Z_+\}.$ Below we will show that star-cyclicity together with the assumption that $\Gamma$ is a pure contraction ensures that the operator $T\ci{\Gamma}$ is what is called a completely non-unitary contraction, meaning that $T\ci \Gamma$ does not have a non-trivial unitary part.  The model theory informs us that such $T\ci{\Gamma}$ is unitarily equivalent to its functional model $\cM_{\te}$, $\te=\te\ci\Gamma$, that is, the compression of the shift operator on the model space $\cK_{\te}$ with the characteristic function $\te=\te\ci\Gamma$ of $T\ci{\Gamma}$.

In this paper we investigate 
the so-called  Clark operator, i.e.~a unitary operator  $\Phi$ that intertwines the contraction $T\ci{\Gamma}$ (in the spectral representation of the unperturbed operator $U$) with its model: $\cM_\te\Phi = \Phi T\ci{\Gamma}$, $\te=\te\ci\Gamma$. The case of rank one perturbations ($d=1$) was treated by D.~Clark when $\theta$ is inner \cite{Clark}, and later by D.~Sarason under the  assumption that $\te$ is an extreme point of the unit ball of $H^\infty$, \cite{SAR}. For finite rank perturbations with inner characteristic matrix-valued functions $\theta$, V.~Kapustin and A.~Poltoratski \cite{KP06} studied boundary convergence of functions in the model space $\cK_\te$. The setting of inner characteristic function corresponds to  the operators $U$ that have purely singular spectrum (no a.c.~component), see e.g.~\cite{DL2013}. 

In \cite{LT15} we completely described the general case  of rank one perturbations  (when the measure can have absolutely continuous part, or equivalently, the characteristic function is not not necessarily inner). 

In the present paper  we extend  the results from \cite{LT15} to finite rank perturbations with general matrix-valued characteristic functions. We first find a universal representation of the adjoint Clark operator, which features a special case of a matrix-valued Cauchy integral operator. By universal we mean that our formula is valid in any transcription of the functional model. 
This representation is a pretty straightforward, albeit more algebraically involved, generalization of the corresponding result from \cite{LT15}; it might look like an ``abstract nonsense'', since it is proved under the assumption that we picked a model operator that  ``agrees'' with the Clark model (more precisely that the corresponding coordinate/parametrizing operators agree). 

However, by careful investigation of the construction of the functional model,  using   the coordinate free Nikolski--Vasyunin model we were able to present a formula giving  the parametrizing operators for the model that agree with given coordinate operators for a general contraction  $T$, see Lemma \ref{l:C-C*}. 
Moreover, for the Sz.-Nagy--Foia\c s transcription of the model we get explicit formulas for the parametrizing operators in terms of the characteristic function, see Lemma \ref{l:C_N-F}; similar formulas can be obtained for other transcriptions of the model.

We also compute the characteristic function of the perturbed operator $T\ci \Gamma$; the formula involves the Cauchy integral of the matrix-valued measure. 

For the Sz.-Nagy--Foia\c s transcription of the model we give a more concrete representation of the adjoint Clark operator in terms of vector-valued Cauchy transform, see Theorem \ref{t-repr3}. This representation looks  more natural when one considers spectral representations of the non-perturbed operator $U$ defined with the help of matrix-valued measures, see Theorem \ref{t:repr04}.



\subsection{Plan of the paper}
In  Section \ref{s-prelims} we set the stage by introducing finite  rank  perturbations and studying some their basic properties. In particular, we discuss the concept of a star-cyclic subspace and find a measure-theoretic characterization for it. 

Main result of Section \ref{s-adjClark} is the universal representation formula for the adjoint Clark operator, see Theorem \ref{t-repr}. In this section we also  introduce the notion of agreement of the coordinate/parametrizing operators and make some preliminary observations about such an agreement. 

Section \ref{s-ModAgree} is devoted to the detailed investigation of the agreement of the coordinate/parametrizing operators. Careful analysis of the construction of the model from the coordinate free point of view of Nikolski--Vasyunin allows us to get for a general contraction $T$ formulas for the parametrizing operators for the model that agree with the coordinate operators, see Lemma \ref{l:C-C*}. 
Explicit formulas (in terms of the characteristic function) are presented for the case of Sz.-Nagy--Foia\c s transcription, see Lemma  \ref{l:C_N-F}. 


The characteristic function $\theta\ci\Gamma$ of the perturbed operator $T\ci \Gamma$ is the topic of Sections \ref{s-charfunc} and \ref{s:PropCharFunct}. Theorem \ref{t-theta} gives a formula for $\theta\ci\Gamma$ in terms of a Cauchy integral of  a matrix-valued measure. In Section \ref{s:PropCharFunct} we show that, similarly to the rank one case, the characteristic functions $\te\ci\Gamma$ and $\te\ci{\OZ}$ are related via a special  linear fractional transformation. Relations between defect functions $\Delta\ci\OZ$ and $\Delta\ci\Gamma$ are also described. 

Section \ref{s-Explanations} contains a brief heuristic overview of what subtle techniques are to come in Sections \ref{s-SIO} and \ref{ss-PhiStarSNF}. 

In Section \ref{s-SIO} we present results about regularizations of  the Cauchy transform, and about uniform boundedness of such generalizations, that we need to get the representation formulas in Section \ref{ss-PhiStarSNF}.

In Section \ref{ss-PhiStarSNF} we give a formula for the adjoint Clark operator in the   Sz.-Nagy--Foia\c s transcription of the model. As in the scalar case the adjoint Clark operator is given by the sum of two terms: one is in essence a vector-valued Cauchy transform (postmultiplied by a matrix-valued function)
, and the second one is just a multiplication operator by a matrix-valued function, see Theorem \ref{t-repr3}. In the case of inner characteristic function (purely singular spectral measure of $U$) the second term disappears, and the adjoint Clark operator is given by what can be considered a matrix-valued analogue of the scalar \emph{normalized Cauchy transform}, see Section \ref{ss:GCT}.   

Section \ref{s:directClark} is devoted to a description of the Clark operator $\Phi$, see Theorem \ref{t:direct Clark}. 

\section{Preliminaries}\label{s-prelims}


Consider the family  of rank $d$ perturbations $U+K$ of a unitary operator $U$ on a separable Hilbert space $H$. If we fix a subspace ${\er}\subset H$, $\dim {\er}=d$ such that $\Ran K\subset {\er}$, then all unitary perturbations of  $U+K$ of $U$ can be parametrized as 
\begin{align}
\label{pert-01}
T= U + (\rx-\bI\ci {\er}) P\ci {\er} U, 
\end{align}
where $\rx$ runs over all possible unitary operators in ${\er}$. 

It is more convenient to factorize the representation of $\rx$ through the fixed space $\fD:=\C^d$ by  picking an isometric operator ${\bB}:\fD
\to H$, $\Ran \bB = {\er}$. 
Then any $\rx$ in \eqref{pert-01} can be  represented as $\rx=\bB \Gamma \bB^*$ where $\Gamma: \fD\to \fD$ (i.e. $\Gamma$ is a $(d\times d)$ matrix). The perturbed operator $T=T\ci\Gamma$ can be rewritten as 
\begin{align}
\label{pert-02}
T= U+ \bB(\Gamma-\bI\ci{\fD}) \bB^*U.
\end{align}
If we decompose the space $H$ treated as the domain as $H= U^*{\er}\oplus (U^*{\er})^\perp$, and the same space treated as the target space as $H={\er}\oplus {\er}^\perp$, then the operator $T$ can be represented with respect to this decomposition as 
\begin{align}
\label{BlDec-01}
T= \left( \begin{array}{cc} \bB\Gamma \bB^*U & 0\\ 0   & T_1 \end{array} \right) ,
\end{align}
where block $T_1$ is unitary. 

From the above decomposition we can immediately see that if $\Gamma$ is a contraction then $T$ is a contraction (and if $\Gamma$ is unitary then $T$ is unitary). 

In this formula we slightly abuse notation, since formally the operator $\bB\Gamma \bB^*U$ is defined on the whole space $H$. However, this operator clearly annihilates $ (U^*{\er})^\perp$, and its range belongs to ${\er}$, so we can restrict its domain and target space to $U^*{\er}$ and ${\er}$ respectively. So when such operators appear in the block decomposition we will assume that its domain and target space are restricted. 

In this paper we assume that the isometry $\bB$ is fixed and that all the perturbations are parametrized by the $(d\times d)$ matrix $\Gamma$.

\subsection{Spectral representation of \texorpdfstring{$U$}{U}}
By the Spectral Theorem the operator $U$ is unitarily equivalent to the multiplication $M_\xi$ by the independent variable $\xi$ in the von Neumann direct integral
\begin{align}
\label{DirInt}
\cH= \int_\T^\oplus E(\xi) \dd \mu(\xi), 
\end{align}
where $\mu$ is a finite Borel measure on $\T$ (without loss of generality we can assume that $\mu$ is a probability measure, $\mu(\T) =1$). 

Let us recall the construction of the direct integral; we present not the most general one, but one that is sufficient for our purposes. Let $E$ be a separable Hilbert space with an orthonormal basis $\{e_n\}_{n=1}^\infty$, and let $N:\T\to \N\cup\{\infty\}$ be a measurable function (the so-called \emph{dimension function}). Define 
\[
E(\xi) = \cspa\{ e_n\in E:1\le n \le N(\xi)\}.
\]
Then the direct integral $\cH$ is the subspace of the $E$-valued space $L^2(\mu;E)=L^2(\T, \mu;E)$ consisting of the functions $f$ such that $f(\xi) \in E(\xi)$ for $\mu$-a.e.~$\xi$.

Note, that the dimension function $N$ and the spectral type $[\mu]$ of $\mu$ (i.e.~the collection of all measures that are mutually absolutely continuous with $\mu$) are spectral invariants of $U$, meaning that they  define operator $U$ up to unitary equivalence. 

So, without loss of generality, we assume that $U$ is the multiplication $M_\xi$ by the independent variable $\xi$ in the direct integral \eqref{DirInt}. 

An important particular case is the case when $U$  is star-cyclic, meaning that there exists a vector $h\in H$ such that $\overline\spa\{U^n h: n\in\Z\}=H$. In this case $N(\xi)\equiv 1$, and the operator $U$ is unitary equivalent to the multiplication operator $M_\xi$  in the  scalar  space $L^2(\mu)=L^2(\T,\mu)$. 

In the representation of $U$ in the direct integral it is convenient to give a ``matrix'' representation of the isometry $\bB$. Namely, for $k=1, 2, \ldots, d$ define functions $b_k\in\cH \subset L^2(\mu;E)$ by $b_k:= \bB e_k $; here $\{e_k\}_{k=1}^d$ is the standard orthonormal basis in $\C^d$. 

In this notation  the operator $\bB$, if we follow the standard rules of the linear algebra is the multiplication by a row  $B$ of vector-valued  functions,  
\[
B(\xi) = (b_1(\xi), b_2(\xi), \ldots, b_d(\xi)).
\] 
If we represent $b_k(\xi)$ in the standard basis in $E$ that we used to construct the direct integral \eqref{DirInt}, then $\bB$ is just the multiplication by the matrix-valued function of size $(\dim E)\times d$. 

\subsection{Star-cyclic subspaces and completely non-unitary contractions} \label{ss-cyccnu}

\begin{defn}\label{d-cyclic}
A subspace ${\er}$  is said to be \emph{star-cyclic} for an operator $T$ on $H$, if
\[
H = \overline{\spa}\{T^k {\er}, (T^*)^k{\er}: k\in\Z_+\}.
\]
\end{defn}
For a  perturbation  (not necessarily unitary) $T=T\ci\Gamma$ of the unitary operator $U$  given by \eqref{pert-02}
the subspace 
\begin{align}
\label{SpanOrbit-01}
\cE=\cspa\{ U^k {\er}, (U^*)^k{\er}: k\in\Z_+\} = \cspa\{ U^k{\er}: k\in\Z\}
\end{align}
is a reducing subspace for both  $U$ and $T\ci\Gamma$ (i.e.~$\cE$ and  $\cE^\perp$ are invariant for both $U$ and $T\ci\Gamma$). 

Since $T\ci \Gamma \bigm|_{\cE^\perp}=U \bigm|_{\cE^\perp}$, the perturbation does not influence the action of $T\ci \Gamma$ on $\cE^\perp$, so nothing interesting for perturbation theory happens on $\cE^\perp$; all action happens on $\cE$. Therefore, we can restrict our attention to $T\ci \Gamma \bigm|_{\cE}$, i.e.~assume without loss of generality that $\er = \Ran \bB$ is a star-cyclic subspace for $U$.

We note that if ${\er}$ is a star-cyclic subspace for $U$ and $\Gamma$ is unitary, then ${\er}$ is also a star-cyclic subspace for all perturbed unitary operators given by \eqref{pert-02}. 

\begin{lem}
\label{l:*cycl}
Let ${\er}=\Ran \bB$ be a star-cyclic subspace for $U$ and let $\Gamma$ be unitary. Then ${\er}$ is also a star-cyclic subspace for all perturbed unitary operators $U\ci\Gamma=T\ci{\Gamma}$ given by \eqref{pert-02}.
\end{lem}


We postpone for a moment a proof of this well-known fact. 
%
%
%

\begin{defn}
\label{d:cnu} A contraction $T$ in a Hilbert space $H$ is called \emph{completely non-unitary}  (c.n.u.~for short) if there is no non-zero reducing subspace on which $T$ acts unitarily. 
\end{defn}

Recall that a contraction is called \emph{strict} if $\|Tx\|<\|x\|$ for all $x\ne\bO$. 
\begin{lem}\label{l-cnu}
If ${\er}=\Ran \bB$ is a star-cyclic subspace for $U$ and $\Gamma$ is a strict contraction, then $T$ defined by \eqref{pert-02} is a c.n.u.~contraction.
\end{lem}

\begin{proof}
Since $\Gamma$ is a strict contraction,  
we get that   ${\bB}\Gamma {\bB}^*U|\ci{U^*{\er}}$ is also a strict contraction. Therefore  \eqref{BlDec-01} implies that 
\begin{align*}
\|Tx\| = \|x\| \quad&\Longleftrightarrow\quad x\perp U^{-1} {\er}\\
\|T^*x\| = \|x\| \quad&\Longleftrightarrow\quad x\perp {\er}.
\end{align*}
Moreover, we can see from \eqref{BlDec-01} that if $x\perp U^{-1}{\er}$ then $Tx= Uf$ and if $x\perp {\er}$ then $T^*x = U^{-1}x$.

Consider a reducing subspace $G$ for $T$ such that $T|\ci{G}$ is unitary. Then the above observations imply $G\perp {\er}$ and $G\perp U^{-1}{\er}$, and that   
for any $x\in G$
\begin{align*}
T^n x = U^n x \qquad\text{as well as}\qquad
\left(T^*\right)^n x = U^{-n}x.
\end{align*}

Since $G$ is a reducing subspace for  $T$ it follows that $U^k x\in G$ for all integers $k$. But this implies that $U^nx\perp {\er}$, or equivalently $x\perp U^n {\er}$  for all $n\in\Z$. But ${\er}$ is a star-cyclic subspace for $U$, so we get a contradiction. 
%
\end{proof}

\begin{proof}[Proof of Lemma \ref{l:*cycl}] 
Assume now that for unitary $\Gamma$, the subspace $\Ran\bB$ is not a star-cyclic subspace for $U\ci\Gamma= T\ci\Gamma$ (but is a star-cyclic subspace for $U$).  Consider the perturbation 
$T_\bO$ 
\[
T\ci\bO = U + \bB(\bO-\bI\ci\fD)\bB^*U. 
\] 
We will show that 
\begin{align}
\label{ChPert-01}
T\ci\bO = U\ci \Gamma + \bB(\bO- \bI\ci\fD)\bB^*U\ci\Gamma
\end{align}
By Lemma \ref{l-cnu} the operator $T\ci\bO$ is a c.n.u.~contraction. 

But, as we discussed in the beginning of this subsection, if $\Ran \bB$ is not star-cyclic for $U$, then for $\cE$ defined by \eqref{SpanOrbit-01} the subspace $\cE^\perp$ is a reducing subspace for $T\ci\Gamma$ (with any $\Gamma$) on which $T\ci\Gamma$ acts unitarily.  

Since by \eqref{ChPert-01} the operator $T_\bO$ is a perturbation of form \eqref{pert-02} of the unitary operator $T\ci\Gamma$, we conclude that the operator $T\ci\bO$ has a non-trivial unitary part, and arrive to a contradiction.  

To prove \eqref{ChPert-01} we notice that 
\begin{align}
\label{ChPert-02}
T\ci\bO= U-\bB\bB^* U = U\ci\Gamma - \bB\Gamma \bB^*U. 
\end{align}
Direct computations show that 
\begin{align*}
U\ci\Gamma U^*\bB = UU^*\bB + \bB(\Gamma-\bI\ci\fD) \bB^*UU^*\bB = \bB + \bB (\Gamma-\bI\ci\fD) =\bB\Gamma. 
\end{align*}
Taking the adjoint of this identity we get that $\bB^*UU^*\ci\Gamma = \bB^*\Gamma^*$, and so $\Gamma \bB^*U = \bB^*U\ci\Gamma$. Substituting $\bB^*U\ci\Gamma$ instead of $\Gamma\bB^*U$ in \eqref{ChPert-02} we get \eqref{ChPert-01}.
\end{proof}

\subsection{Characterization of star-cyclic subspaces}

Recall that for an isometry $\bB:\cD \to \cH$ (where $\cH$ is the direct integral \eqref{DirInt}) we denoted by  $b_k\in\cH$ the ``columns'' of $\bB$,  
\[
b_k=\bB e_k,
\]
where $e_1, e_2, \ldots, e_d$ is the standard basis in $\C^d$. 

\begin{lem}
\label{l:cycl-02} Let $U$ be the multiplication $M_\xi$ by the independent variable $\xi$ in the direct integral $\cH$ given by \eqref{DirInt}, and let  $\bB:\C^d\to \cH$ be as above. The space $\Ran \bB=\spa\{b_k:1\le k\le d\}$ is star-cyclic for $U$ if and only if $\cspa\{b_k(\xi) : 1\le k \le d\} = E(\xi)$ for $\mu$-a.e.~$\xi$.
\end{lem}
\begin{proof}
First assume that $\Ran \bB$ is not a star-cyclic subspace for $U$. Then there exists $f\in \cH\subset L^2(\mu;E)$, $f\ne 0$ $\mu$-a.e., such that
\begin{align*}
U^l f\perp b_k
\qquad\text{for all }l\in \Z, \text{ and } k=1,\hdots,d, 
\end{align*}
or, equivalently
\begin{align*}
\int_\T \Bigl( f(\xi) , {b_k(\xi)}\Bigr)\ci E \xi^l \dd \mu(\xi) = 0 \qquad\text{for all }l\in \Z, \text{ and } k=1,\hdots,d. 
\end{align*}
But that means for all $k=1, 2, \ldots, d$ we have
\begin{align*}
\Bigl( f(\xi),  b_k(\xi) \Bigr)\ci E= 0 \qquad \mu\text{-a.e.}, 
\end{align*}
so on some set of positive $\mu$ measure (where $f(\xi)\ne \bO$) we have 
\begin{align}
\label{Span-ne-E}
\cspa\{b_k(\xi): 1\le k\le d\}\subsetneqq E(\xi).
\end{align}

Vice versa, assume that \eqref{Span-ne-E} holds on some Borel subset $A\subset\T$ with $\mu(A)>0$. 
For $n=1, 2, \ldots, \infty$ define sets
$
A_n:=\{\xi\in A: \dim E(\xi)=n\}
$.
Then $\mu(A_n)>0$ for some $n$. Fix this $n$ and denote the corresponding space $E(\xi)$, $\xi\in A_n$ by $E_n$. 

We know that $\cspa\{b_k(\xi): 1\le k\le d\}\subsetneqq E_n$ on $A_n$, so there exists $e\in E_n$ such that 
\[
e\notin \cspa\{b_k(\xi): 1\le k\le d\}
\]
on a set of positive measure in $A_n$. 

Trivially, if $f\in \cspa\{U^k \Ran\bB: k\in\Z\}$ then 
\[
f(\xi)\in \cspa\{b_k(\xi): 1\le k\le d\}\qquad \mu\text{-a.e.}, 
\]
and therefore $f=\1\ci{A_n}e$ is not in $\cspa\{U^k \Ran\bB: k\in\Z\}$. 
\end{proof}

\subsection{The case of star-cyclic \texorpdfstring{$U$}{U}}
If $U$ is star-cyclic (i.e.~it has a one-dimensional star-cyclic subspace/vector), $U$ is unitarily equivalent to the multiplication operator $M_\xi$ in the scalar space $L^2(\mu)$; of course the scalar space $L^2(\mu)$ is a particular case of the direct integral, where all spaces $E(\xi) $  are one-dimensional. 

In our general vector-valued case, Lemma \ref{l:cycl-02} says that $\Ran \bB$ is star-cyclic for $U$ if and only if there is no measurable set $A$, $\mu(A)>0$, on which all the functions $b_k$ vanish. 
So, we know that $U$ has a star-cyclic vector. Here we ask the question:
\begin{center}
\emph{Does operator $U$ have a star-cyclic vector that belongs to a prescribed (finite-dimensional)  star-cyclic subspace?}
\end{center}

The following lemma answers ``yes" to that question. Moreover, it implies that if $\Ran \bB$ is star-cyclic for $U=M_\xi$ on the scalar-valued space $L^2(\mu)$, then almost all vectors  $b\in \Ran \bB$ are star-cyclic for $U$. As the result is measure-theoretic in nature, we formulate it in a general context.

\begin{lem}
\label{l:cycl-03}
Consider a $\sigma$-finite scalar-valued measure $\tau$ on  a measure space $\cX$. Let $b_1, b_2, \ldots, b_d\in L^2(\tau)$ be such that 
\begin{align*}
\sum_{k=1}^d |b_k| \ne 0 \qquad\tau\text{-a.e.}
\end{align*}
Then for almost all (with respect to the Lebesgue measure) $\alpha=(\alpha_1, \alpha_2, \ldots, \alpha_d)\in\C^d$ we have
\begin{align*}
\sum_{k=1}^d \alpha_k b_k \ne 0 \qquad\tau\text{-a.e.~on }\cX.
\end{align*}
\end{lem}

\begin{rem*} 
 The above lemma also holds for almost all $\alpha\in\R^d$.
\end{rem*}

\begin{proof}[Proof of Lemma \ref{l:cycl-03}]
Consider first the case   $\tau(\cX)<\infty$.

We proceed by induction in $d$. Clearly, if $|b_1|\neq 0$ $\tau$-a.e.~on $\cX$, then $\alpha b_1\neq 0$ $\tau$-a.e.~on $\cX$ for all $\alpha\in \C\setminus\{0\}$.

Now assume the statement of the Lemma for $d=n$ for some $n\in \N$. Deleting a set of $\tau$-measure $0$, we can assume that  $\sum_{k=1}^{n+1} |b_k| \ne 0$ on $\cX$. 

Let $\cY:=\{x\in\cX: \sum_{k=1}^n |b_k(x)|>0\}$. 
By the induction assumption for almost all $\alpha'=(\alpha_1, \alpha_2, \ldots, \alpha_n)$ 
\begin{align*}
b(\alpha',x) : =\sum_{k=1}^n \alpha_k b_k(x) \ne 0 \qquad {on\ } \cY. 
\end{align*}
Fix $\alpha'=(\alpha_1, \alpha_2, \ldots, \alpha_n)$ such that $b(\alpha', x) \ne 0$ on $\cY$.  We will show that for any such fixed $\alpha'$ the measure
\begin{align}
\label{CountMany-alpha}
\tau\left(\left\{x\in\cX: \sum_{k=1}^{n+1} \alpha_k b_k(x) =0\right\} \right)>0
\end{align}
only for countably many values of $\alpha_{n+1}$.

To show that define for $\beta=\alpha_{n+1} \in \C$ the set
\[
\cX_{\beta}:=\left\{x\in \cX:b(\alpha',x)+\beta b_{n+1}(x)=0\right\}.
\]

Let $\widetilde\beta \in \C\setminus\{0\}$, $\wt\beta\ne \beta$. We claim that the sets $X_\beta$ and $X_{\widetilde\beta}$ are disjoint. 

Indeed, the assumption that $\sum_{k=1}^{n+1} |b_k|>0$ implies that $b_{n+1}\ne 0$ on $\cX\setminus\cY$, so $\cX_\beta, \cX_{\wt\beta}\in\cY$. Moreover, solving for $b_{n+1}$ we get that if $\beta\ne0$, then 
\begin{align*}
\cX_\beta= \{x\in\cY: b_{n+1}(x) = -b(\alpha', x)/\beta\}, 
\end{align*}
and similarly for $\cX_{\wt\beta}$. Since $b(\alpha',x)\ne0$ on $\cY$, we get that 
\begin{align*}
b(\alpha', x)/\beta \ne b(\alpha', x)/\wt\beta\qquad \forall x\in\cY, 
\end{align*}
so if $\beta\ne0$, then $\cX_\beta$ and $\cX_{\wt\beta} $ are disjoint as preimages of disjoint sets (points). 

If $\beta=0$, then $\cX_0=\cX\setminus\cY$, so the sets $\cX_{\wt\beta}$ and $\cX_0$ are disjoint.

The set $\cX$ has finite measure, and $\cX$ is the union of disjoint sets $\cX_\beta$, $\beta\in \C$. So, only countably many sets $\cX_\beta$ can satisfy $\tau(\cX_\beta)>0$. 
We have proved the lemma for $\tau(\mathcal X)<\infty$.

The rest can be obtained by Tonelli's theorem. Namely, define
\[
\cA:= \left\{(x,\alpha): x\in\cX, \alpha\in\C^{n+1} , \sum_{k=1}^{n+1}\alpha_k b_k(x) =0\right\}
\]
and let $F=\1\ci\cA$. From the Tonelli Theorem we can see that 
\begin{align}
\label{Int1_A}
\int\1\ci\cA (x,\alpha) \dd m(\alpha) \dd\tau(x)>0
\end{align}
if and only if for the set of $\alpha\in\C^{n+1}$ of positive Lebesgue measure
\[
\tau\left( \left\{x\in\cX: \sum_{k=1}^{n+1}\alpha_k b_k(x) =0 \right\}\right) >0 .
\]
It follows from \eqref{CountMany-alpha} that for almost all $\alpha'=(\alpha_1, \alpha_2, \ldots, \alpha_n)\in\C^n$
\[
\int\1\ci\cA (x,\alpha', \alpha_{n+1}) \dd m(\alpha_{n+1}) \dd\tau(x) = 0,
\]
so, by Tonelli, the integral in \eqref{Int1_A} equals $0$. 
%
\end{proof}

\section{Abstract formula for the adjoint Clark operator}\label{s-adjClark}

In this section we introduce necessary known facts about functional models and then give a general abstract formula for the adjoint Clark operator.  To do this we need a new notion of coordinate/parametrizing operators for the model and their agreement: the abstract representation formula (Theorem   \ref{t-repr}) holds under the assumption that the coordinate operators $\bC$ and $\bC_*$ agree with the Clark model. 

Later in Section  \ref{s-ModAgree} we construct the coordinate operators that agree with the Clark, and in Section \ref{s-charfunc} we compute the characteristic function, so the abstract Theorem  \ref{t-repr} will give us concrete, albeit complicated formulas.  


\subsection{Functional models
}
\begin{defn}
\label{d:defects}
Recall that for a contraction $T$ its \emph{defect operators} $D\ci T$ and $D\ci{T^*}$ are defined as 
\begin{align*}
D\ci{T} := (\bI -T^*T)^{1/2}, \qquad D\ci{T^*}:= (\bI-TT^*)^{1/2}. 
\end{align*}
The \emph{defect spaces} $\fD\ci T$ and $\fD\ci{T^*}$ are defined as 
\begin{align*}
\fD\ci{T}:=\clos\Ran D\ci T, \qquad  \fD\ci{T^*}:=\clos\Ran D\ci{T^*}. 
\end{align*}
\end{defn}
The characteristic function is an (explicitly computed from the contraction $T$) op\-erator-valued function $\theta\in H^\infty(\fD\tto \fD_*)$, where $\fD$ and $\fD_*$ are Hilbert spaces of appropriate dimensions, 
\[
\dim \fD=\dim \fD\ci{T}, \qquad \dim \fD_*=\dim \fD\ci{T^*}\,.
\]

Using the characteristic function $\theta$ one can then construct the so-called \emph{model space} $\cK_\theta$, which is a subspace of a weighted $L^2$ space $L^2(\T, W; \fD_*\oplus\fD)= L^2(W; \fD_*\oplus\fD)$ with an operator-valued weight $W$. The model operator $\cM_\theta:\cK_\theta \to \cK_\theta$ is then defined as the \emph{compression} of the multiplication $M_z$ by the independent variable $z$, 
\begin{align*}
\cM_\theta f = P\ci{\cK_\theta} M_z f, \qquad f\in\cK_\theta; 
\end{align*}
here $M_z f(z) =zf(z)$. 

Let as remind the reader, that the norm in the weighted space $L^2(\T,W;H)$ with an operator weight $W$ is given by 
\[
\|f\|\ci{L^2(W;H)}^2 = \int_\T (W(z) f(z), f(z) )\ci H \dd m(z) ;
\]
in the case $\dim H=\infty$ there are some technical details, 
but in the finite-dimensional case considered in this paper everything is pretty straightforward.

The best known example of a model is  the Sz.-Nagy--Foia\c{s} (transcription of a) model,  \cite{SzNF2010}. The Sz.-Nagy--Foia\c{s} model space $\cK_\te$ is a subspace of a non-weighted space $L^2(\fD_*\oplus\fD)$ (the weight $W\equiv\bI$), 
given by
\begin{align*}
\cK_\te
:=
\begin{pmatrix}H^2(\fD_*)\\\clos\Delta L^2(\fD)\end{pmatrix}
\ominus \begin{pmatrix}\theta\\\Delta\end{pmatrix} H^2(\fD),
\end{align*}
where
\[
\Delta(z) := (\OID_\mathfrak{D}-\theta^\ast(z)\theta(z))^{1/2}
\quad
\text{and}\quad
\begin{pmatrix}\theta\\\Delta\end{pmatrix} H^2(\fD)= 
\left\{ \left(\begin{array}{c}\theta f \\ \Delta f\end{array}\right) : f\in H^2(\fD) \right\}.
\]

In literature, the case when the vector-valued characteristic function $\theta$ is \emph{inner} (i.e.~its boundary values are isometries for a.e.~$z\in\T$) is often considered. Then $\Delta(z) = \OZ$ on $\T$, so in that case the second component of $\cK_\te$ collapses completely and the Sz.-Nagy--Foia\c{s} model space reduces to the familiar space
\[
\cK_\te
=
H^2(\fD_*)\ominus \theta H^2(\fD).
\]

Also, in the literature, cf \cite{SzNF2010}, the characteristic function is defined up to multiplication by constant unitary factors from the right and from the left. Namely, two functions $\te\in H^\infty(\fD\to\fD_*)$ and  $\wt\te \in H^\infty(\wt\fD\to\wt\fD_*)$ are equivalent if there exist unitary operators $U:\fD\to\wt\fD$ and $U_*:\fD_*\to\wt\fD_*$ such that $\wt\te = U_*\te U^*$. 

It is a well-known fact, cf \cite{SzNF2010}, that two c.n.u.~contractions are unitarily equivalent if and only if their characteristic functions are equivalent as described above. So, usually in the literature the characteristic function was understood as the corresponding equivalence class, or an arbitrary representative in this class. However, in this paper, to get  correct formulas it is essential to track which representative is chosen.

%
%
%
%

\subsection{Coordinate operators, parameterizing operators, 
and their agreement 
}
\label{s:agree-01}

Let $T:H\to H$ be a contraction, and let $\fD$, $\fD_*$ be Hilbert spaces, $\dim \fD= \dim\fD\ci T$, $\dim \fD_* = \dim \fD\ci{T^*}$. Unitary operators $V:\fD\ci{T}\to \fD$ and $V_*:\fD\ci{T^*}\to \fD_*$ will be called \emph{coordinate operators} for the corresponding defect spaces; the reason for that name is that often spaces $\fD$ and $\fD_*$ are spaces with a fixed orthonormal basis (and one can introduce coordinates there), so the operators introduce coordinates on the defect spaces. 

The inverse operators $V^*:\fD\to \fD\ci{T}$ and $V_*^*:\fD_*\to \fD\ci{T^*}$ will be called \emph{parameterizing} operators. For a contraction $T$ we will use symbols $V$ and $V_*$ for the coordinate operators, but for its model $\cM_\te$ the parametrizing operators will be used, and we  reserve letters $\bC$ and $\bC_*$ for these operators. 

Let $T$ be a c.n.u.~contraction with characteristic function $\theta\in H^\infty(\fD\tto \fD_*)$, and let $\cM_\theta: \cK_\theta\to \cK_\theta$ be its model. Let also $V:\fD\ci{T}\to \fD$ and $V_*:\fD\ci{T^*}\to \fD_*$ be coordinate operators for the defect spaces of $T$, and 
$\bC:\fD\ci{\cM_\theta}\to \fD$ and $\bC_*:\fD\ci{\cM_\theta^*}\to \fD_*$ be the parameterizing operators for the defect spaces of $\cM_\theta$ (this simply means that all 4 operators are unitary). 

We say that the operators $V$, $V_*$ agree with operators $\bC$, $\bC_*$ if there exists a unitary operator $\Phi:\cK_\theta\to H$ intertwining $T$ and $\cM_\theta$, 
\[
T\Phi = \Phi \cM_\theta,
\] 
and such that 
\begin{align}
\label{agree-01}
\bC^* = V\Phi\Bigm|_{\fD\ci{\cM_\theta}}, \qquad \bC_*^* = V_*\Phi\Bigm|_{\fD\ci{\cM_\theta^*}}\,.
\end{align}
The above identities simply mean that the diagrams below are commutative. 
\begin{center}
\begin{tikzpicture}[>=angle 90]
\matrix(a)[matrix of math nodes,
row sep=2.5em, column sep=2.5em,
text height=1.5ex, text depth=0.25ex]
{D\ci{T}&&\fD&\fD_*&&D\ci{T^*}\\
&&&&&\\
&\fD\ci{\cM_\te}&&&\fD\ci{\cM_\te^*}&\\};
\path[->](a-1-1) edge node[auto] {$V$} (a-1-3);
\path[->](a-1-4) edge node[auto] {$ V_*^*$} (a-1-6);
\path[->](a-3-2) edge node[auto] {$\Phi$} (a-1-1);
\path[<-](a-1-3) edge node[auto] {$\bC$} (a-3-2);
\path[<-](a-1-6) edge node[auto] {$\Phi$} (a-3-5);
\path[->](a-3-5) edge node[auto] {$\bC_*$} (a-1-4);
\end{tikzpicture}
\end{center}

In this paper, when convenient, we always extend an operator between subspaces to the operator between the whole spaces, by extending it by $0$ on the orthogonal complement of the domain;
slightly abusing notation we will use the same symbol for both operators. Thus a unitary operator between subspaces $E$ and $F$ can be treated as a partial isometry with initial space $E$ and final space $F$, and vice versa.  With this agreement \eqref{agree-01} can be rewritten as 
\begin{align*}
\bC^* = V\Phi, \qquad \bC_*^* = V_*\Phi . 
\end{align*}

\subsection{Clark operator} Consider a contraction $T$ given by \eqref{pert-02} with $\Gamma$ being a strict contraction. We also assume that $\Ran \bB$ is a star-cyclic subspace for $U$, so $T$ is a c.n.u.~contraction, see Lemma \ref{l-cnu}.

We assume that $U$ is given in its spectral representation, so $U$ is the multiplication operator $M_\xi$ in the direct integral $\cH$. 


A Clark operator $\Phi:\cK_\theta\to \cH$ is a unitary operator, intertwining this special contraction $T$  and its model $\cM_\theta$, $\Phi\cM_\theta = T\Phi$, or equivalently
\begin{align}
\label{e-intertwinePhi*}
\Phi^* T
=
 \cM_\te \Phi^*.
\end{align}
We name it so after D.~Clark, who in \cite{Clark} described it for rank one perturbations of unitary operators with purely singular spectrum. 

We want to describe the operator $\Phi$ (more precisely, its adjoint $\Phi^*$) in our situation.  
In our case, $\dim \fD\ci{T} = \dim \fD\ci{T^*}=d$, and it will be convenient for us to consider models with $\fD=\fD_*=\C^d$. 

As it was discussed above, it can be easily seen from the representation \eqref{BlDec-01} that the operators    $U^*\bB:\fD=\C^d\to \fD\ci{T}$ and $\bB:\fD=\C^d\to \fD\ci{T^*}$ are unitary operators canonically (for our setup) identifying $\fD$ with the corresponding defect spaces, i.e.~the canonical parameterizing operators for these spaces. The  corresponding coordinate operators are given by $V=\bB^* U$, $V_*=\bB^*$. 


We say that parametrizing operators $\bC:\fD\to \fD\ci{\cM_\theta}$, $\bC_* : \fD\to \fD\ci{\cM_\theta^*}$ \emph{agree} with the Clark model, if the above coordinate operators $V=\bB^* U$, $V_*=\bB^*$ agree with the parametrizing operators $\bC$, $\bC_*$ in the sense of Subsection \ref{s:agree-01}. In other words, they agree
if there exists a Clark operator $\Phi$ such that the following diagram commutes. 

\begin{equation}
\label{d:agree-02}
\begin{tikzpicture}[>=angle 90, baseline={([yshift=-.5ex]current bounding box.center)}]
\matrix(a)[matrix of math nodes,
row sep=2.5em, column sep=2.5em,
text height=1.5ex, text depth=0.25ex]
{\fD\ci{T}&&\fD = \C^d&&\fD\ci{T^*}\\
&&&&\\
&\fD\ci{\cM_\theta}&&\fD\ci{\cM_\theta^*}&\\};
\path[->](a-1-1) edge node[auto] {$\bB^* U$} (a-1-3);
\path[->](a-1-3) edge node[auto] {$ \bB$} (a-1-5);
\path[<-](a-3-2) edge node[auto] {$\Phi^*$} (a-1-1);
\path[->](a-1-3) edge node[auto] {$\bC$} (a-3-2);
\path[->](a-1-5) edge node[auto] {$\Phi^*$} (a-3-4);
\path[<-](a-3-4) edge node[auto] {$\bC_*$} (a-1-3);
\end{tikzpicture}
\end{equation}

Note, that in this diagram one can travel in both directions: to change the direction one just needs to take the adjoint of the corresponding operator.

%
%
%
Slightly abusing notation, we use $\bC$ to also denote the extension of $\bC$ to the model space $\cK\ci\te$ by the zero operator, and similarly for $\bC_*$.

Note that agreement of $\bC$ and $\bC_*$ with the Clark model can be rewritten as 
\begin{align}\label{e-comm}
\Phi^* (\bB^* U)^*= \bC, \qquad
\Phi^*\bB = \bC_*.
\end{align}

And by taking restrictions (where necessary) we find
\begin{align}\label{e-INTER}
\cM_\theta \bC = \bC_*\Gamma
\qquad
\text{and}
\qquad
\cM_\theta^* \bC_* = \bC\Gamma^*.
\end{align}

We express the action of the model operator and its adjoint in an auxiliary result. The result holds in any transcription of the model. 
We will need the following simple fact. 

\begin{lem}
\label{l:Tdefect}
For a contraction $T$
\begin{align*}
T \fD\ci T \subset \fD\ci{T^*}, \qquad T^* \fD\ci{T^*} \subset \fD\ci{T} \,.
\end{align*}
\end{lem}
\begin{proof} Since $D\ci T $ is a strict contraction on $\fD\ci T$ we get that 
\begin{align}
\notag
\|T x\|=\|x\|\quad &\iff \quad x\perp \fD\ci{T}, \\
\intertext{and similarly, since $T^*$ is a strict contraction on $\fD\ci{T^*}$,}
\label{IsomT-02}
\|T^* x\|=\|x\|\quad &\iff \quad x\perp \fD\ci{T^*} \,.
\end{align}
Thus the operator $T$ is an isometry on $\fD\ci{T}^\perp$, so the  polarization identity implies that $T^*Tx=x$ for all $x\in \fD\ci{T}^\perp$. Together with \eqref{IsomT-02} this implies that $T(\fD\ci{T}^\perp) \subset \fD\ci{T^*}^\perp$, which is equivalent to the inclusion $T^* \fD\ci{T^*} \subset \fD\ci{T}$\,. 

Replacing $T$ by $T^*$ we get $T \fD\ci{T} \subset \fD\ci{T^*}$.  
\end{proof}

%
%
%
%
%

\begin{lem}\label{l-model}
Let $T$ be as defined in \eqref{pert-02} with $\Gamma$ being a strict contraction. Assume also that 
$\Ran \bB$ is star-cyclic (so $T$ is completely non-unitary, see Lemma \ref{l-cnu}).

Let $\theta\in H^\infty(\fD\tto \fD_*)$, $\fD=\fD_*=\C^d$, be the characteristic function of $T$, and let $\cM_\theta:\cK_\theta\to \cK_\theta$ be a model operator. Let $\bC:\fD\to \fD\ci{\cM_\theta}$ and $\bC_* :\fD\to \fD\ci{\cM_\theta^*}$ be the parametrizing unitary operators, that agree with a Clark model. 

Then 
\[
\cM_\theta = M_z +(\bC_* \Gamma - M_z \bC) \bC^*
\qquad\text{and}\qquad
\cM_\theta^* = M_{\bar z}+ (\bC  \Gamma^*  - M_{\bar{z}} \bC_*)\bC_*^* .
\]
\end{lem}

\begin{proof}
Since operator $\cM_\theta$ acts on $\cK_\theta\ominus\fD\ci{\cM_\theta}$ as the multiplication operator $M_z$, we can trivially write
\[
\cM_\theta = M_z (\bI - P\ci{\fD\ci{\cM_\te}}) + \cM_\theta P\ci{\fD\ci{\cM_\te}} .
\]
Recalling that $\bC:\fD\to\cK_\theta$ is an isometry with range $\fD\ci{\cM_\te}$, we can see that $P\ci{\fD\ci{\cM_\te}} = \bC \bC^* $, so
\begin{align}
\label{MzP-MzCC}
M_z (\bI - P\ci{\fD\ci{\cM_\te}})= M_z (\bI - \bC \bC^*) .
\end{align}
 Using the identity $P\ci{\fD\ci{\cM_\te}} = \bC \bC^* $ and the first equation of \eqref{e-INTER} we get
\[
\cM_\te P\ci{\fD\ci{\cM_\te}} = \cM_\theta \bC\bC^* 
=
\bC_*\Gamma\bC^*,
\]
which together with  \eqref{MzP-MzCC} gives us the desired formula for $\cM_\theta$. 

To get the formula for $\cM_\theta^*$ we represent it as 
\[
\cM_\theta^* = M_{\overline z} (\bI - P\ci{\fD\ci{\cM_\te^*}}) + \cM_\theta^* P\ci{\fD\ci{\cM_\te^*}}.
\]
Using the identities 
\[
P\ci{\fD\ci{\cM_\te^*}} = \bC_*\bC_*^* , \qquad \cM_\theta^* P\ci{\fD\ci{\cM_\te^*}} = \bC \Gamma^* \bC_*^*
\]
(the first holds because $\fD\ci{\cM_\te^*}$ is the range of the isometry $\bC_*$, and the second one follows from the second equation in \eqref{e-INTER}),  we get the desired formula. 
\end{proof}

\subsection{Representation Theorem}
For a (general) model operator $\cM_\te$, $\te\in H^\infty(\fD\to \fD_*)$,  
the parametrizing operators $\bC:\fD\to \fD\ci{\cM_\theta}$, $\bC_* : \fD_*\to \fD\ci{\cM_\theta^*}$ give rise to (uniquely defined) operator-valued functions $C$ and $C_*$, where $C(\xi):\fD\to \fD\oplus \fD_*$, $C_*(\xi):\fD_*\to \fD\oplus \fD_*$ and
\begin{align}\label{d-Cz}
(\bC e)(z)
& =
C(z) e 
&&\text{for all}
\quad e\in \fD, 
\\
\label{d-C*z}
(\bC_* e_*)(z)
&=
C_*(z) e_* 
&&\text{for all}
\quad e_*\in \fD_* 
. 
\end{align}
If we fix orthonormal bases in $\fD$ and $\fD_*$, then the $k$th column of the matrix of $C(\xi)$ is defined as $(\bC_*e_k)(\xi)$, where $e_k$ it the $k$th vector in the orthonormal basis in $\fD$, and similarly for $C_*$. 

If $\cM_\te$ is a model for a contraction $T=T\ci\Gamma$ with $\Gamma$ being a strict contraction on $\fD=\C^d$, we can see from \eqref{BlDec-01} that $\dim \fD\ci T = \dim\fD\ci{T^*}=d$, so we can always pick a characteristic function $\theta\in H^\infty(\fD\to\fD_*)$ (i.e.~with $\fD_*=\fD=\C^d$).


The following formula for the adjoint  $\Phi^*$ of the Clark operator $\Phi$  generalizes  the ``universal" representation theorem \cite[Theorem 3.1]{LT15} to higher rank perturbations.

\begin{theo}[Representation Theorem]\label{t-repr}
Let $T$ be as defined in \eqref{pert-02} with $\Gamma$ being a strict contraction and $U=M_\xi$ in $\cH\subset L^2(\mu;E)$. 
Let $\te =\te\ci{T}$ be a characteristic function of $T$,  and let  $\cK_\te$ and $\cM_\te$ be the corresponding model space and model operator. 

Let $\bC:\fD\to \fD\ci{\cM_\theta}$ and $\bC_* :\fD\to \fD\ci{\cM_\theta^*}$ be the parameterizing unitary operators%
\footnote{Note that here we set $\fD_*=\fD$, which is possible because the dimensions of the defect spaces are equal.}
that agree with Clark model, i.e.~such that  \eqref{e-comm} is satisfied for some Clark operator $\Phi$. And let $C(z)$ and $C_*(z)$ be given by \eqref{d-Cz} and \eqref{d-C*z}, respectively.
%
%
%
%

Then  the action of the adjoint  Clark operator $\Phi^*$ is given by 
\begin{align}
\label{e-repr}
\bigl(\Phi^* hb\bigr)(z)
=
h(z)C_*(z){\bB}^*b
+
(C_*(z)-z C(z))
\int\ci\T
\frac{h(\xi) - h(z)}{1-z\bar\xi} B^*(\xi)b(\xi) 
\dd\mu(\xi),
\end{align}
for any $b\in\Ran \bB$ and for all $h\in C^1(\T)$;
here 
\[
B^*(\xi) = \left( \begin{array}{c} b_1(\xi)^* \\ b_2(\xi)^* \\ \vdots \\ b_d(\xi)^* \end{array}\right)
\] 
and 
$
\bB^* b
=
\int\ci\T B^*(\xi) b(\xi) d\mu(\xi),
$ as explained more thoroughly in the proof below. 
\end{theo}

\begin{rem*}
The above theorem looks like an abstract nonsense, because right now it is not clear how to find the parametrizing operators $\bC$ and $\bC_*$ that agree with the Clark model. However, Theorem \ref{t-theta} below  gives an explicit formula for the characteristic function $\theta$ (one of the representative in the equivalence class), and  
 Lemma \ref{l:C_N-F}  gives an explicit formulas for  $\bC$ and $\bC^*$ in the Sz.-Nagy--Foia\c s transcription, that agree with Clark model for our $\theta$.
\end{rem*}


When $d=1$ this formula agrees with the special case of the representation formula derived in \cite{LT15}. While some of the ideas of the following proof were originally developed there, the current extension to rank $d$ perturbations requires several new ideas and a more abstract way of thinking.

\begin{proof}[Proof of Theorem \ref{t-repr}]
Recall that $U= M_\xi$,
so $T =  M_\xi+ {\bB}(\Gamma - \OID\ci{\C^d}){\bB}^* M_\xi$. The intertwining relation $\Phi^* T=\cM_\te\Phi^*$ then can be rewritten as 
\begin{align}\label{equation}
\Phi^* M_\xi
+
\Phi^*{\bB}(\Gamma - \OID\ci{\C^d}){\bB}^* U
=
\Phi^* T
=
\cM_\te\Phi^*
=
[M_z +(\bC_* \Gamma - M_z \bC) \bC^*]\Phi^*;
\end{align}
here we used Lemma \ref{l-model} to express the model operator in the right hand side of \eqref{equation}. 

 By the commutation relations in equation \eqref{e-comm}, the term $\Phi^*{\bB}\Gamma{\bB}^* U$ on the left hand side of \eqref{equation} cancels with the term $\bC_* \Gamma  \bC^*\Phi^*$ on the right hand side of \eqref{equation}. Then \eqref{equation} can be rewritten as 
\begin{align}
\notag
\Phi^* M_\xi
&=
M_z \Phi^*
+\Phi^*{\bB} \OID\ci{\C^d}{\bB}^* U - M_z \bC \bC^*\Phi^*
\\ \label{e-PhiUn-01}
&=
M_z \Phi^*
+(\bC_*-M_z\bC){\bB}^* M_\xi ;
\end{align}
the last identity holds because, by \eqref{e-comm}, we have $\Phi^*\bB =\bC_*$ and $\bC^*\Phi^* = {\bB}^* U =\bB^* M_\xi $. 

%
 
Right multiplying \eqref{e-PhiUn-01} by $M_\xi$ and using \eqref{e-PhiUn-01} we get
 \begin{align*}
\Phi^* M_\xi^2
&=
M_z \Phi^*M_\xi+(\bC_*- M_z \bC){\bB}^*M_\xi^2
\\&=
M_z^2 \Phi^*+M_z (\bC_*- M_z \bC){\bB}^*M_\xi+ (\bC_*- M_z \bC){\bB}^*M_\xi^2 .
 \end{align*}
Right multiplying the above equation by $M_\xi$ and using \eqref{e-PhiUn-01} again we get the identity 
%
\begin{align}\label{e-PhiUn}
\Phi^* M_\xi^n
=
M^n_z\Phi^*
+
\sum_{k=1}^n M_z^{k-1}  (\bC_* - M_z \bC)
 {\bB}^* M_\xi^{n-k+1},
\end{align}
with $n=3$. Right multiplying by $M_\xi$ and applying \eqref{e-PhiUn-01} we get by induction that \eqref{e-PhiUn} holds for all $n\ge 0$. 
(The case $n=0$ trivially reads $\Phi^*=\Phi^*$, and equation \eqref{e-PhiUn-01} is precisely the case $n=1$.)


 
%
We now apply \eqref{e-PhiUn} to some $b\in\Ran \bB$. By commutative diagram \eqref{d:agree-02} we get that $\Phi^* b =\bC_*\bB^* b$, i.e.~$(\Phi^* b)(z) = C_*(z) \bB^*b$.  Using this identity we  get
%
 \begin{align}\label{e-application}
\bigl( \Phi^* M_\xi^n b \bigr)(z)
& =
z^n ( \Phi^* b)(z)
+
\sum_{k=1}^n
z^{k-1}(C_*(z)- z C(z) )  {\bB}^* M_\xi^{n-k+1} b
\\ \notag
& = z^n C_*(z)( \bB^* b)(z)
+
(C_*(z)- z C(z) )  \sum_{k=1}^n
z^{k-1} {\bB}^* M_\xi^{n-k+1} b .
\end{align}
To continue, we recall that $\bB:\C^d\to L^2(\mu;E)$ acts as multiplication by matrix $B(\xi) = (b_1(\xi), b_2(\xi), \ldots, b_d(\xi))$, so its adjoint $\bB^*:\cH\subset L^2(\mu;E)\to \C^d$ is given by
\[
\bB^* f
=
\int\ci\T B^*(\xi) f(\xi) d\mu(\xi)
\qquad\text{for }f\in \cH,
\]
where the integral can be expanded as
\[
\int\ci\T B^*(\xi) f(\xi) d\mu(\xi)
=
\begin{pmatrix}
\int\ci\T  b_1(\xi)^* f(\xi) d\mu(\xi)\\
\int\ci\T  b_2(\xi)^* f(\xi) d\mu(\xi)\\
\vdots\\
\int\ci\T  b_d(\xi)^* f(\xi) d\mu(\xi)\\
\end{pmatrix}.
\]
Using the sum of geometric progression formula we evaluate the sum in \eqref{e-application} to
\begin{align}
\notag
\sum_{k=1}^n
z^{k-1}  {\bB}^* M_\xi^{n-k+1} b
&=
\sum_{k=1}^n
z^{k-1} \int\ci\T \xi^{n-k+1} B^*(\xi)  b(\xi) 
d\mu(\xi)
\\
\notag&=
 \int\ci\T \sum_{k=1}^n
 z^{k-1} \xi^{n-k+1} B^*(\xi)  b(\xi) 
d\mu(\xi)
\\
\label{e-geometric}&= 
\int\ci\T
\frac{\xi^n - z^n}{1-z\bar\xi}  
B^*(\xi)  b(\xi) 
d\mu(\xi) .
\end{align}

Thus, we have proved \eqref{e-repr} for monomials $h(\xi)= \xi^n$,  $n\ge 0$. And by linearity of $\Phi^*$ the representation \eqref{e-repr} holds for (analytic) polynomials $h$ in $\xi$.

The argument leading to determine the action of $\Phi^*$ on polynomials $h$ in $\bar \xi$ is similar. But we found that the devil is in the details and therefore decided to include much of the argument.

First observe that the intertwining relation \eqref{e-intertwinePhi*} is equivalent to $\cM_\te^*\Phi^*=\Phi^* T^*$. Recalling $T^* = U^* + U^* \bB (\Gamma^*- \OID\ci{\C^d}) \bB^*$ and the resolution of the adjoint model operator $\cM_\te^*$ (see second statement of Lemma \ref{l-model}), we obtain
\[
M_{\bar z}\Phi^*+ (\bC  \Gamma^*  - \bar{z} \bC_*)\bC_*^* \Phi^*=
\cM_\te^*\Phi^*
=\Phi^* T^*
=
\Phi^* U^* - \Phi^* U^* \bB (\Gamma^*- \OID\ci{\C^d}) \bB^*.
\]

The terms involving $\Gamma^*$ on the left hand side and the right hand side cancel by the commutation relations in equation \eqref{e-comm} (actually by their adjoints). Now, rearrangement and another application of the adjoints of the commutation relations in equation \eqref{e-comm} yields
\begin{align}
\notag
\Phi^* M_{\bar\xi}
&=
\Phi^* U^*
=
M_{\bar z}\Phi^*
+ 
 \Phi^* U^* \bB  \OID\ci{\C^d} \bB^* - \bar{z} \bC_*\bC_*^* \Phi^*
 =
M_{\bar z}\Phi^*
+ 
(\bC  - M_{\bar{z}}\bC_* ) \bB^*\\
\label{e-mxi}
&=
M_{\bar z}\Phi^*
+ 
M_{\bar{z}}(M_{{z}}\bC  - \bC_* ) \bB^*.
\end{align}
In analogy to the above, we right multiply \eqref{e-mxi} by $M_{\bar \xi}$ and apply \eqref{e-mxi} twice to obtain
\[
\Phi^* M_{\bar\xi}^2
=
M_{\bar z}^2\Phi^*
+ 
  \sum_{k=1}^2
M_{\bar z}^{k} (M_z \bC-\bC_*) 
{\bB}^*M_{\bar \xi}^{2-k}
.
\]
Inductively, we conclude
\[
\Phi^* M_{\bar\xi}^n
=
M^n_{\bar z}\Phi^*
-
\sum_{k=1}^n
M_{\bar z}^{k} (\bC_*-M_{z} \bC) {\bB}^*M_{\bar\xi}^{n-k},
\]
which \emph{differs in the exponents and in the sign from its counterpart expression} in equation \eqref{e-PhiUn}.

Through an application of this identity to $b$ and by the commutative diagram \eqref{d:agree-02}, we see
 \begin{align*}
\bigl( \Phi^* M_{\bar \xi}^n b \bigr)(z)
& =
\bar z^n ( \Phi^* b)(z)
-
\sum_{k=1}^n
\bar z^{k}(C_*(z)- z C(z) )  {\bB}^* M_{\bar\xi}^{n-k} b
\\ \notag
& = \bar z^n C_*(z)( \bB^* b)(z)
-
(C_*(z)- z C(z) )  \sum_{k=1}^n
\bar z^{k} {\bB}^* M_{\bar\xi}^{n-k} b .
\end{align*}
As in equation \eqref{e-geometric}, but here with the geometric progression
\[
-\sum_{k=1}^n (\bar z)^k (\bar \xi)^{n-k} = \frac{(\bar \xi)^n - (\bar z)^n}{1-\bar \xi z},
\]
we can see equation \eqref{e-repr} for monomials $\bar \xi^n$, $n\in \N$.
And by linearity of $\Phi^*$, we obtain the \emph{same formula} \eqref{e-repr} for functions $h$ that are polynomials in $\bar\xi$.

We have proved \eqref{e-repr} for trigonometric polynomials $f$. The theorem now follows by a standard approximation argument, developed in \cite{LT09}. The application of this argument to the current situation is a slight extension of the one used in \cite{LT15}.
Fix $f\in C^1(\T)$ and let $\{p_k\}$ be a sequence of trigonometric polynomials with uniform on $\T$ approximations $p_k\rightrightarrows f$ and  $p_k'\rightrightarrows f'$.
In particular, we have $|p_k'|$ is bounded (with bound independent of $k$) and $p_k\to f$ as well as $p_kb\to fb$ in $L^2(\mu;E)$. Since $\Phi^*$ is a unitary operator, it is bounded and therefore we have convergence on the left hand side 
$\Phi^* p_k b \to \Phi^* f b$
in $\cK_\te$.

To investigate convergence on the right hand side, first recall that the model space is a subspace of the weighted space $L^2(W;\fD_*\oplus \fD)$.

So convergence of the first term on the right hand side happens, since $p_k\rightrightarrows f$ and the operator norm $\|\bC_* \bB^*\| = 1$ implies $p_kC_*(z) \bB^*b=p_k\bC_* \bB^*b\to f\bC_* \bB^*b=fC_*(z) \bB^*b$ in $\cK_\te$.

Lastly, to see convergence of the second term on the right hand side, consider auxiliary functions $f_k:=f-p_k$. We have $f_k\rightrightarrows 0$ and $f_k'\rightrightarrows 0$. Let $I_{\xi, z}\subset \T$ denote the shortest arc connecting $\xi $ and $z$. Then by the intermediate value theorem
\begin{align*}
 |f_k(\xi )- f_k(z)| 
 \le \|f_k'\|_\infty |I_{\xi,z}|
  \qquad\text{for all } \xi,z\in\T.
\end{align*}
In virtue of the geometric estimate $|I_{\xi,z}|\le\frac\pi2 |\xi-z|$, we obtain
\[
\left| \frac{f_k(\xi)-f_k(z)}{1-\overline \xi z} \right| \le \frac\pi2 \|f_k'\|_\infty \to 0 \qquad \text{as } k\to \infty.
\]
And since $\bB^*$ is bounded as a partial isometry, we conclude the componentwise uniform convergence
\begin{align*}
\int\frac{p_k(\xi)-p_k(z)}{1-\overline \xi z}B^*(\xi)  b(\xi) \, d\mu(\xi)
\quad 
\rightrightarrows 
\quad
\int\frac{f(\xi)-f(z)}{1-\overline\xi z}B^*(\xi)  b(\xi) \, d\mu(\xi)
\qquad
z\in \T.
\end{align*}
By Lemma \ref{l:bound-C} below the functions $W^{1/2}C$ and $W^{1/2}C_*$ are bounded, and so is the function $W^{1/2}C_1$, $C_1(z) := C_*(z) - z C(z)$. That means the multiplication operator $f\mapsto C_1 f$ is a bounded operator $L^2(\fD)\to L^2(W;\fD_*\oplus \fD)$ (recall that in our case $\fD=\fD_*$ and we use $\fD_*$ here only for the consistency with the general model notation). 

The uniform convergence implies the convergence in $L^2(\fD)$, so the boundedness of the multiplication by $C_1$ implies the convergence in norm in the second term in the right hand side of \eqref{e-repr} (in the norm of $L^2(W;\fD_*\oplus\fD)$). 
\end{proof}

\section{Model and agreement of operators}\label{s-ModAgree}
We want to explain how to get operators $\bC$ and $\bC_*$ that agree with each other. 

To do that we need to understand in more detail how the model is constructed, and what operator gives the unitary equivalence of the function and its model. 

Everything starts with the notion of unitary dilation. Recall that for a contraction $T$ in a Hilbert space $H$ its unitary dilation is a unitary operator $\cU$ on a bigger space $\cH$, $H\subset \cH$ such that for all $n\ge 0$
\begin{align}
\label{dila-01}
T^n = P\ci H \cU^n \bigm|_{H}. 
\end{align}
Taking the adjoint of this identity we immediately get that 
\begin{align}
\label{dila-02}
(T^*)^n = P\ci H \cU^{-n} \bigm|_{H}.
\end{align}
A dilation is called \emph{minimal} if it is impossible to replace $\cU$ by its restriction to a reducing subspace and still have the identities \eqref{dila-01} and \eqref{dila-02}. 

The structure of minimal unitary dilations is well known.

\begin{theo}[{\cite[Theorem 1.4]{Nik-Vas_model_MSRI_1998}} and {\cite[Theorem 1.1.16]{Nik-book-v2}}]
\label{t:MUDil}
Let $\cU:\cH\to\cH$ be a minimal unitary dilation of a contraction $T$. Then $\cH$ can be decomposed as  $\cH = G_*\oplus H \oplus G$, and with respect to this decomposition $\cU$ can be represented as 
\begin{align}
\label{MUDil}
\cU = \left( 
\begin{array}{ccc}
 \cE_*^* &  0 & 0  \\
 D\ci{T^*}V_*^* & T  &  0 \\
 -V T^* V_*^* &  V D\ci{T} &   \cE
\end{array}
\right)
\end{align}
where $\cE:G\to G$ and $\cE_*:G_*\to G_*$ are pure isometries, $V$ is a partial isometry with initial space $\fD_{T}$ and the final space $\ker \cE^*$ and $V_*$ is a partial isometry with initial space $\fD\ci{T^*}$ and final space $\ker \cE_*^*$. 

Moreover, any minimal unitary dilation of\, $T$ can be obtained this way. Namely if we pick auxiliary Hilbert spaces $G$ and $G_*$ and isometries $\cE$ and $\cE_*$ there  with $\dim \ker \cE^* = \dim\fD\ci{T}$, $\dim\ker \cE_*^* = \dim\fD\ci{T*}$ and then pick arbitrary partial isometries $V$ and $V_*$ with initial and final spaces as above, then \eqref{MUDil} will give us a minimal unitary dilation of $T$.   
\end{theo}

The construction of the model then goes as follows. We take auxiliary Hilbert spaces $\fD$ and $\fD_*$, $\dim \fD = \dim\fD\ci{T}$, $\dim \fD_* = \dim\fD\ci{T*}$, and construct operators $\cE$ and $\cE_*$ such that $\ker \cE^* =\fD$, $\ker\cE_*^* =\fD_*$. We can do that by putting $G=\ell^2(\fD)=\ell^2(\Z_+;\fD)$, and defining $\cE (x_0, x_1, x_2, \ldots ) = (0, x_0, x_1, x_2, \ldots)$, $x_k\in \fD$, and similarly for $\cE_*$. 

Picking arbitrary partial isometries $V$ and $V_*$ with initial and final spaces as in the above Theorem \ref{MUDil} we get a minimal unitary dilation $U$ of $T$ given by \eqref{MUDil}.

\begin{rem*}
Above,  we were speaking a bit informally, by identifying $x\in \fD$ with the sequence $(x, 0, 0, 0, \ldots) \in \ell^2(\fD)$, and $x_*\in \fD_*$ with $(x_*, 0, 0, 0, \ldots) \in \ell^2(\fD)$. 

To be absolutely formal, we need to define canonical embeddings $\be: \fD\to G=\ell^2(\fD)$, $\be_*:\fD_*\to G_*=\ell^2(\fD_*)$ with
\begin{align}
\label{eq:e}
\be(x) &:=(x, 0, 0, 0, \ldots), \qquad x\in \fD, \\
\label{eq:e_*}
\be_*(x_*) &:=(x_*, 0, 0, 0, \ldots), \qquad x\in \fD_*  .
\end{align}
Then, picking arbitrary unitary operators $V:\fD\ci{T}\to \fD$, $V_*:\fD\ci{T^*}\to \fD_*$, we rewrite \eqref{MUDil} to define the corresponding unitary dilation as 
\begin{align}
\label{MUDil-01}
\cU = \left( 
\begin{array}{ccc}
 \cE_*^* &  0 & 0  \\
 D\ci{T^*}V_*^* \be_*^* & T  &  0 \\
 -\be V T^* V_*^*\be_*^* &  \be V D\ci{T} &   \cE
\end{array}
\right)  \,.
\end{align}

The reason for being so formal is that if $\dim \fD\ci{T}=\dim\fD\ci{T^*}$it is often convenient to put $\fD=\fD_*$, but we definitely want to be able to distinguish between the cases when $\fD$ is identified with  $\ker \cE$ and when with $\ker\cE_*$. 
\end{rem*}

%
%
%

\medskip

We then  define \emph{functional embeddings} $\pi:L^2(\fD)\to \cH$ and $\pi_*: L^2(\fD_*)\to \cH$ by 
\begin{align*}
\pi \left(\sum_{k\in\Z} z^k e_k \right) & = \sum_{k\in\Z}  \cU^k  \be (e_k), \qquad e_k\in \fD, \\
\pi_* \left(\sum_{k\in\Z} z^k e_{k} \right) & = \sum_{k\in\Z}  \cU^{k+1}  \be_*(e_{k}), \qquad e_{k}\in \fD_* .
\end{align*}
We refer the reader to \cite[Section 1.6]{Nik-Vas_model_MSRI_1998} or to \cite[Section 1.2]{Nik-book-v2} for the details. Note that there $\fD$ and $\fD_*$ were abstract spaces, $\dim \fD = \dim\ker\cE^*$ and $\dim \fD_*=\dim\ker\cE_*^*$, and the unitary operators $v:\fD\to\ker \cE^*$, $v_* :\fD_* \to \ker \cE_*^*$ used in the formulas there are just the canonical embeddings $\be$ and $\be_*$ in our case. 

Note that $\pi$ and $\pi_*$ are isometries. 

Note also that for $k\ge0$
\begin{align*}
\cU^k \be(e) &= \cE^k e,  &&e\in \fD,\\
\cU^{-k} \be_*(e_*) & = \cE_*^k e_*, && e_*\in \fD_*, 
\end{align*}
so
\begin{align*}
\pi (H^2(\fD)) = G, \qquad \pi_* (H^2_-(\fD_*)) = G_* .
\end{align*}

The characteristic function is then defined as follows. We consider
the operator $\btheta=\pi_*^*\pi :L^2(\fD)\to L^2(\fD_*)$. It is easy to check that $M_z\btheta = \btheta M_z$, so the $\btheta$ is a multiplication by a function $\theta\in L^\infty(\fD\tto \fD_*)$. It is not hard to check that $\btheta $ is a contraction, so $\|\theta\|_\infty\le 1$.  Since 
\begin{align*}
\pi(H^2(\fD)) = G \perp G_* = \pi_*(H^2_-(\fD_*)),
\end{align*}
we can conclude that $\theta\in H^\infty(\fD\tto \fD_*)$. 

The characteristic function $\theta=\theta\ci T$ can be explicitly computed, see \cite[Theorem 1.2.10]{Nik-book-v2},
\begin{align}
\label{CharFunction-01}
\theta_T(z) = V_*\left( -T + z D_{T^*}\left(\I\ci\cH - z T^*\right)^{-1} D_T \right) V^* \Bigm|_{\fD}, \qquad z\in \D. 
\end{align}
Note that the particular representation of $\theta$ depends on the coordinate operators $V$ and $V_*$ identifying defect spaces $\fD\ci{T}$ and $\fD\ci{T^*}$ with the abstract spaces $\fD$ and $\fD_*$. 

To construct a model (more precisely its particular transcription) we need to construct a unitary map $\Psi$ between the space $\cH$ of the minimal unitary dilation $\cU$ and its spectral representation. 

Namely, we represent $\cU$ as a multiplication operator in some subspace $\wt\cK=\wt\cK_\theta$ of $L^2(\fD_*\oplus \fD)$ or its weighted version. 

We need to construct a unitary operator $\Psi:\cH\to \wt \cK$ intertwining $\cU$ and $M_z$ on $\wt\cK$, i.e. such that 
\begin{align}
\label{eq:Psi-01}
\Psi \cU = M_z \Psi . 
\end{align}

Note that if $T$ is a completely non-unitary contraction, then $\pi(L^2(\fD)) + \pi_* (L^2(\fD_*))$ is dense in $\cH$. 

So, for $\Psi$ to be unitary it is necessary and sufficient that $\Psi^*$ acts isometrically on  $\pi (L^2(\fD))$ and on $\pi_*( L^2(\fD_*))$, and that for all $f\in L^2(\fD)$, $g\in L^2(\fD_*)$
\begin{align}
\label{pi^pi*}
(\Psi^* \pi f , \Psi^*\pi g)\ci{\wt\cK} = (\pi f, \pi_* g)\ci{\cH}   =(\btheta f, g)\ci{L^2(\fD_*)} ;
\end{align}
the last equality here is just the definition of $\btheta$. 

Of course, we need $\Psi^*$ to be onto, but that can be easily accomplished by  restricting the target space $\wt\cK$ to $\Ran \Psi^*$.

Summing up, we have:

\[
\begin{array}{ccccccc}
\cH&=&G&\oplus& H &\oplus&G_*\vspace{.4cm}\\\vspace{.4cm}
\,\,\,\phantom{\hat{\varphi}}\Bigg\downarrow{\Psi^*}
&&\quad\,\,\phantom{\hat{\varphi}}\Bigg\downarrow{\Psi^*|\ci{G}}
&&\quad\quad\,\phantom{\hat{\varphi}}\Bigg\downarrow{\Psi^*|\ci{H}}
&&\,\,\quad\phantom{\hat{\varphi}}\Bigg\downarrow{\Psi^*|\ci{G_*}}\\
\widetilde\cK&=&\mathcal{G}&\oplus&\cK_\te&\oplus&\mathcal{G}_*
\end{array}
\]

\subsection{Pavlov transcription}
Probably the easiest way to construct the model is to take $\wt \cK$ to be the weighted space $L^2(\fD_*\oplus\fD , W)$ where the weight $W$ is picked to make the simplest operator $\Psi^*$ to an isometry, and  is given by
\begin{align}
\label{PavlovWeight}
W (z)= \left(\begin{array}{cc}\bI\ci{\fD_*} & \theta (z)\\ \theta(z)^* & \bI\ci{\fD} \end{array}\right).
\end{align}
Now operator $\Psi^*$ is defined on $\pi(L^2(\fD))$ and on $\pi_* (L^2(\fD_*))$ as 
\begin{align}
\notag
\Psi^*\Bigl(\sum_{k\in\Z}\cU^k \be(e_k) \Bigr) & = \sum_{k\in\Z} z^k 
\left(\begin{array}{c}0 \\e_k \end{array}\right), \qquad  e_k\in\fD,    \\
\label{Psi*Pavlov-02}
\Psi^*\Bigl(\sum_{k\in\Z}\cU^k \be_*(e_{k}) \Bigr) & = \sum_{k\in\Z} z^{k-1} 
\left(\begin{array}{c}e_k \\ 0\end{array}\right) , \qquad e_{k}\in\fD_*\,,
\end{align}
or equivalently
\begin{align*}
\Psi^* (\pi f ) & =  
\left(\begin{array}{c}0 \\ f \end{array}\right), \qquad  f\in L^2(\fD),    \\
\Psi^*(\pi_* f) & =  
\left(\begin{array}{c} f \\ 0\end{array}\right) , \qquad f \in L^2(\fD_*)\,,
\end{align*}
The incoming and outgoing spaces $\cG_*= \Psi^* G_*$, $\cG= \Psi^* G$ are given by 
\begin{align*}
\cG_* :=  \clos\ci{\wt\cK}\left\{\left(\begin{array}{c} f \\ 0 \end{array} \right): f\in H^2_-(\fD_*) \right\},  \qquad 
\cG :=  \clos\ci{\wt\cK}\left\{\left(\begin{array}{c} 0 \\ f \end{array} \right): f\in H^2(\fD) \right\}, 
\end{align*}
and the model space $\cK=\cK_\theta$ is defined as 
\[
\cK_\theta= \wt\cK \ominus (\cG_*\oplus\cG). 
\]

\subsection{Sz.-Nagy--\texorpdfstring{Foia\c{s}}{Foias} transcription}
This transcription appears when one tries to make the operator $\Psi^*$ to act into a non-weighted space $L^2(\fD_*\oplus\fD)$. We make the action of the operator $\Psi^*$ on $\pi_*(L^2(\fD_*))$ as simple as possible, 
\begin{align}
\label{Psi*N-F-01}
\Psi^*\Bigl(\sum_{k\in\Z}\cU^k \be_*(e_{k}) \Bigr) & = \sum_{k\in\Z} z^{k-1} 
\left(\begin{array}{c}e_k \\ 0\end{array}\right), \qquad e_{k}\in\fD_*
\end{align}
(this is exactly as in \eqref{Psi*Pavlov-02}). Action of $\Psi^*$ on $\pi(L^2(\fD))$ is defined as
\begin{align}
\label{Psi*N-F-02}
\Psi^*\Bigl(\sum_{k\in\Z}\cU^k \be(e_{k}) \Bigr) & = \sum_{k\in\Z} z^{k} 
\left(\begin{array}{c}\theta e_k \\ \Delta e_k\end{array}\right), \qquad e_{k}\in\fD\,,
\end{align}
where $\Delta(z) = (\bI - \theta(z)^*\theta(z))^{1/2}$. The equations \eqref{Psi*N-F-01} and \eqref{Psi*N-F-02} can clearly be rewritten as 
\begin{align}
\label{Psi*N-F-03}
\Psi^* (\pi f ) & =  
\left(\begin{array}{c}\theta f \\ \Delta f \end{array}\right), \qquad  f\in L^2(\fD),    \\
\label{Psi*N-F-04}
\Psi^*(\pi_* f) & =  
\left(\begin{array}{c} f \\ 0\end{array}\right) , \qquad f \in L^2(\fD_*)\, . 
\end{align}
Note, that $\theta$ in the top entry in \eqref{Psi*N-F-02} and \eqref{Psi*N-F-03} is necessary to get 
\eqref{pi^pi*}; after \eqref{Psi*N-F-01} (equivalently \eqref{Psi*N-F-04}) is chosen, one does not have any choice here. The term $\Delta$ in the bottom entry of \eqref{Psi*N-F-02} and \eqref{Psi*N-F-03} is there to make $\Psi^*$ act isometrically on $\pi(L^2(\fD))$. There is some freedom here; one can left multiply $\Delta$ by any operator-valued function $\phi$ such that $\phi(z)$ acts isometrically on $\Ran \Delta(z)$. However, picking just $\Delta$ is the canonical choice for   the Sz.-Nagy--Foia\c{s} transcription, and we will follow it. 

The incoming and outgoing spaces are given by 
\begin{align*}
\cG_* := \left(\begin{array}{c} H^2_-(\fD_*) \\ 0\end{array}\right), \qquad
\cG := \left(\begin{array}{c}\theta \\ \Delta\end{array}\right) H^2(\fD) 
. 
\end{align*}

The model space is given by 
\begin{align}
\label{N-F-K_theta}
\cK_\theta := 
\left(\begin{array}{c} L^2(\fD_*) \\ \clos \Delta L^2(\fD) \end{array}\right) \ominus 
(\mathcal{G}_*\oplus\mathcal{G})
=
\left(\begin{array}{c} H^2(\fD_*) \\ \clos \Delta L^2(\fD) \end{array}\right) \ominus 
\left(\begin{array}{c}\theta \\ \Delta\end{array}\right) H^2(\fD) . 
\end{align}

\begin{rem*}
While the orthogonal projection from
\[
\left(\begin{array}{c} L^2(\fD_*) \\ \clos \Delta L^2(\fD) \end{array}\right)
\quad\text{to}\quad
\left(\begin{array}{c} L^2(\fD_*) \\ \clos \Delta L^2(\fD) \end{array}\right)\ominus \mathcal{G}_*
\]
is rather simple, the one from
\[
\left(\begin{array}{c} L^2(\fD_*) \\ \clos \Delta L^2(\fD) \end{array}\right)
\quad\text{to}\quad
\left(\begin{array}{c} L^2(\fD_*) \\ \clos \Delta L^2(\fD) \end{array}\right)\ominus \mathcal{G}
\]
involves the range of a Toeplitz operator.
\end{rem*}

\subsection{De Branges--Rovnyak transcription} This transcription looks most complicated, but its advantage is that both coordinates are analytic functions. To describe this transcription, we use the auxiliary weight $W = W(z)$ as in the Pavlov transcription, see \eqref{PavlovWeight}. 
The model space is the subspace of $L^2(\fD_*\oplus \fD, W^{[-1]})$, where for a self-adjoint operator $A$ the symbol $A^{[-1]}$ denotes its Moore--Penrose (pseudo)inverse, i.e.~$A^{[-1]}=0$ on $\Ker A$ and $A^{[-1]}$ is the left inverse of $A$ on $(\Ker A)^\perp$. 

The operator $\Psi^*: \cH \to L^2(\fD_*\oplus \fD, W^{[-1]})$ is defined by 
\begin{align*}
\Psi^* (\pi f ) & =  W \left(\begin{array}{c} 0 \\ f \end{array}\right) =
\left(\begin{array}{c}  \theta f\\  f\end{array}\right), \qquad  f\in L^2(\fD),    \\
\Psi^*(\pi_* f) & =  W \left(\begin{array}{c} f \\ 0\end{array}\right) =
\left(\begin{array}{c}  f \\ \theta^* f\end{array}\right) , \qquad f \in L^2(\fD_*)\, . 
\end{align*}
The incoming and outgoing spaces are
\begin{align*}
\cG_*:= \left(\begin{array}{c}  \bI \\ \theta^* \end{array}\right) H^2(\fD_*), \qquad 
\cG:= \left(\begin{array}{c}  \theta \\ \bI \end{array}\right) H^2(\fD),  
\end{align*}
and the model space is defined as 
\begin{align*}
\cK_\theta := \left\{  \left(\begin{array}{c}  f \\ g \end{array}\right) : f\in H^2(\fD_*),\  g\in H^2_-(\fD), \ g-\theta^* f \in \Delta L^2(\fD) \right\},
\end{align*}
see \cite[Section 3.7]{Nik-Vas_model_MSRI_1998} for the details (there is a typo in \cite[Section 3.7]{Nik-Vas_model_MSRI_1998}, in the definition of $\cK_\theta$ on p.~251 it should be $f\in H^2(E_*)$, $g\in H^2(E)$) .

\subsection{Parametrizing operators for the model, agreeing with coordinate operators}
The parametrizing operators that agree with the coordinate operators $V$ and $V_*$ are described in the following lemma, which holds for any transcription of the model. 

Let $T$ be a c.n.u.~contraction, and let $V:\fD\ci{T} \to \fD$ and $V_*:\fD\ci{T^*} \to \fD_*$ be coordinate operators for the defect spaces of $T$. Let $\theta=\theta\ci{T} = \theta\ci{T, V, V_*}\in H^\infty(\fD\tto \fD_*)$ be the characteristic function of $T$, defined by \eqref{CharFunction-01}, and let  
$\cM_\theta$ be the corresponding model operator (in any transcription).

Recall that $\Psi $ is a unitary operator intertwining the minimal unitary dilation $\cU$ of $T$ and the multiplication operator $M_z$ in the corresponding function space, see \eqref{eq:Psi-01}. The operator $\Psi$ determines transcription of the model, so for any  particular transcription it is known. 

Define
\begin{align}
\label{eq:tilde e}
\wt \be := \Psi^*\be, \qquad \wt\be_* := \Psi^* \be_*, 
\end{align}
where the embedding $\be$ and $\be_*$ are defined by \eqref{eq:e}, \eqref{eq:e_*}. 

\begin{lem}
\label{l:C-C*}
Under the above assumptions the parametrizing operators $\bC_*:\fD_*\to\fD\ci{\cM_\te^*}$ and $\bC:\fD\to\fD\ci{\cM_\te}$ given by 
\begin{align}
\label{C*_CoFree}
\bC_* e_* &= \left(D\ci{\cM_\theta^*}\bigm|_{\fD_{\cM_\theta^*}}\right)^{-1} P\ci{\cK_\theta} M_z \wt\be_* (e_*) , \qquad e_*\in \fD_*, 
\\
\label{C_CoFree}
\bC e &= \left(D\ci{\cM_\theta}\bigm|_{\fD_{\cM_\theta}}\right)^{-1} P\ci{\cK_\theta} M_{\bar z} \wt\be (e) , \qquad e\in\fD,  
\end{align}
agree with the coordinate operators $V$ and $V_*$. 
\end{lem}
\begin{rem*}It follows from the equation \eqref{MUDil-02}  below that $P\ci{\cK_\theta} M_z \wt\be_* (e_*)  \in \Ran D\ci{\cM_\theta^*}$ as well as $P\ci{\cK_\theta} M_{\bar z} \wt\be (e) \in \Ran D\ci{\cM_\theta}$, so everything in \eqref{C*_CoFree}, \eqref{C_CoFree} is well defined. 
\end{rem*}

\begin{proof}[Proof of Lemma \ref{l:C-C*}]
Right and left multiplying \eqref{MUDil-01} by $\Psi$ and $\Psi^*$ respectively, we get
\begin{align}
\label{MUDil-02}
\Psi^* \cU \Psi = \left( 
\begin{array}{ccc}
 \wt\cE_*^* &  0 & 0  \\
 D\ci{\cM_\theta^*}\bC_* \wt\be_*^* & \cM_\theta  &  0 \\
 -\wt\be \bC^* \cM_\theta^* \bC_*\wt\be_*^* &  \wt\be \bC^* D\ci{\cM_\theta} &   \wt\cE
\end{array}
\right)  \,, 
\end{align}
where $\wt \cE= \Psi^* \cE\Psi$, $\wt \cE_*=  \Psi \cE_*\Psi$,  $\bC^* =  V\Psi$, $\bC_*^* =  V_*\Psi$, $\wt \be= \Psi^*\be$, $\wt \be_*=\Psi^*\be_*$. 

The operators $\wt\be$ and $\wt\be_*$ are the canonical embeddings of $\fD$ and $\fD_*$ into $\cG$ and $\cG_*$ that agree with the canonical embeddings $\be$ and $\be_*$. The operators $\bC$ and $\bC_*$ are the parameterizing operators for the defect spaces of the model operator $\cM_\theta$ that agree with the coordinate operators $V$ and $V_*$ for the defect spaces of the operator $T$. 

In any particular transcription of the model, the operator $\Psi^*\cU\Psi$ is known (it is just the multiplication by $z$ in an appropriate function space), so we get from the decomposition  \eqref{MUDil-02} 
\begin{align*}
D\ci{\cM_\theta^*}\bC_* \wt\be_*^* = P\ci{\cK_\theta} M_z \Bigm|_{\cG_*} , \qquad
D\ci{\cM_\theta}\bC \wt\be^*  = P\ci{\cK_\theta} M_{\bar z} \Bigm|_{\cG_*} .
\end{align*}
Right and left multiplying the first identity by $\be_*$ and $\Bigl(D\ci{\cM_\theta^*}\bigm|_{\fD_{\cM_\theta^*}}\Bigr)^{-1}$ respectively, we get \eqref{C*_CoFree}. Similarly, to get \eqref{C_CoFree} we just  right and left multiply the second identity by $\be$ and $\left(D\ci{\cM_\theta}\bigm|_{\fD_{\cM_\theta}}\right)^{-1}$. 
%
%
%
\end{proof}

Applying the above Lemma \ref{l:C-C*} to a particular transcription of the model, we can get  more concrete formulas for $\bC$, $\bC_*$ just in terms of characteristic function $\theta$. 
For example, the following lemma gives formulas for $\bC$ and $\bC_*$ in the Sz.-Nagy--Foia\c{s} transcription.

\begin{lem}\label{l:C_N-F}
Let $T$ be a c.n.u.~contraction, and let $\cM_\theta $ be its model in Sz.-Nagy--Foia\c{s} transcription, with the characteristic function $\theta=\theta\ci{T, V, V_*}$, $\theta\in H^\infty(\fD\tto\fD_*)$. 

Then the maps $\bC_*:\fD_*\to\fD\ci{\cM_\te^*}$ and $\bC:\fD\to\fD\ci{\cM_\te}$ given by
\begin{align}\label{C*_N-F}
\bC_* e_* &=
\kf{ \OID-{\te} (z){\te}^\ast(0)}{ -\Delta (z){\te}^\ast(0) } \left( \OID- {\te}(0){\te}^\ast(0) \right)^{-1/2} e_*,& e_*&\in \fD_*,
\\
\label{C_N-F}
\bC e & =
\kf{ z^{-1} \left({\te}(z)- {\te}(0)\right)}{ z^{-1} \Delta (z)}\left( \OID- {\te}^\ast(0){\te}(0) \right)^{-1/2} e,&e&\in \fD,
\end{align}
agree with the coordinate operators $V$ and $V_*$. 
\end{lem}

\begin{proof}
To prove \eqref{C*_N-F} we will use \eqref{C*_CoFree}. It follows from \eqref{Psi*N-F-01} that 
\[
\wt\be_*(e_*)=z^{-1} \left(\begin{array}{c} e_* \\0\end{array}\right), 
\]
so by \eqref{C*_CoFree}
\begin{align}
\label{C*_N-F-02}
\bC_* e_* = (\bI - \cM_\theta\cM_\theta^*)\bigm|_{\fD_{\cM_\theta^*}}^{-1/2} P\ci{\cK_\theta} \left(\begin{array}{c} e_* \\0\end{array}\right),\qquad e_*\in\fD_*. 
\end{align}
It is not hard to show that 
\begin{align}
\label{P_theta_N-F}
P\ci{\cK_\theta}  \left(\begin{array}{c} e_* \\0\end{array}\right) = \left(\begin{array}{c} \bI - \theta\theta(0)^* \\ -\Delta\theta(0)^*\end{array}\right) e_*\,.
\end{align}
One also can compute
\begin{align}
\label{comm-defects-00}
(\bI-\cM_\theta\cM_\theta^*)  \left(\begin{array}{c} f \\ g\end{array}\right) = \left(\begin{array}{c} \bI - \theta\theta(0)^* \\ -\Delta\theta(0)^*\end{array}\right) f(0), \qquad \left(\begin{array}{c} f \\ g\end{array}\right)\in\cK_\theta. 
\end{align}
Combining the above identities we get that 
\begin{align}
\label{comm-defects-01}
(\bI-\cM_\theta\cM_\theta^*) P\ci{\cK_\theta}  \left(\begin{array}{c} e_* \\0\end{array}\right) 
=\left(\begin{array}{c} \bI - \theta\theta(0)^* \\ -\Delta\theta(0)^*\end{array}\right) 
(e_*- \theta(0)\theta^*(0) e_*).
\end{align}
As we discussed above just after \eqref{C_CoFree}, $P\ci{\cK_\theta}  \left(\begin{array}{c} e_* \\0\end{array}\right)\in \Ran D\ci{\cM_\theta^*}$, so in \eqref{comm-defects-01}  we can replace 
$(\bI-\cM_\theta\cM_\theta^*)$ by its restriction onto $\fD\ci{\cM_\theta^*}$. 

Applying $(\bI-\cM_\theta\cM_\theta^*)\bigm|_{\fD\ci{\cM_\theta^*}}$ to \eqref{comm-defects-01} (with $(\bI-\cM_\theta\cM_\theta^*)$ replaced by its restriction onto $\fD\ci{\cM_\theta}$)  and using \eqref{comm-defects-00} we get
\begin{align*}
\bigl( (\bI-\cM_\theta\cM_\theta^*)\bigm|_{\fD_{\cM_\theta^*}} \bigr)^2 P\ci{\cK_\theta}  \left(\begin{array}{c} e_* \\0\end{array}\right) 
=\left(\begin{array}{c} \bI - \theta\theta(0)^* \\ -\Delta\theta(0)^*\end{array}\right) 
\bigl(\bI\ci{\fD_*}- \theta(0)\theta^*(0) \bigr)^2 e_*
\end{align*}
Applying $(\bI-\cM_\theta\cM_\theta^*)\bigm|_{\fD\ci{\cM_\theta^*}}$  to the above identity, and using again \eqref{comm-defects-00}, we get by induction that  
%
%
\begin{align}
\label{comm-defects-02}
\f\bigl( (\bI-\cM_\theta\cM_\theta^*)\bigm|_{\fD_{\cM_\theta^*}} \bigr) P\ci{\cK_\theta}  \left(\begin{array}{c} e_* \\0\end{array}\right) 
=\left(\begin{array}{c} \bI - \theta\theta(0)^* \\ -\Delta\theta(0)^*\end{array}\right) 
\f\bigl(\bI\ci{\fD_*}- \theta(0)\theta^*(0) \bigr) e_*
\end{align}
for any monomial $\f$, $\f(x) =x^n$, $n\ge0$ (the case $n=0$ is just the identity \eqref{P_theta_N-F}). 

 Linearity implies that \eqref{comm-defects-02} holds for any polynomial $\f$. 
Using standard approximation reasoning we get that $\f$ in \eqref{comm-defects-02} can be any measurable function. In particular, we can take $\f(x) =x^{-1/2}$, which together with \eqref{C*_N-F-02}   gives us \eqref{C*_N-F}.

To prove \eqref{C_N-F} we proceed similarly. Equation \eqref{Psi*N-F-02} implies that 
\begin{align*}
\wt\be(e) = \left(\begin{array}{c}\theta\\ \Delta \end{array}\right) e, 
\end{align*}
so by \eqref{C_CoFree}
\begin{align}
\label{C_N-F-02}
\bC e = \bigl((\bI - \cM\ci\theta^*\cM\ci\theta)\bigm|_{\fD_{\cM_\theta}}\bigr)^{-1/2} P\ci{\cK_\theta} M_{\bar z} \left(\begin{array}{c} \theta \\ \Delta \end{array}\right) e,\qquad e\in\fD.
\end{align}
One can see that 
\begin{align*}
P\ci{\cK_\theta} M_{\bar z} \left(\begin{array}{c} \theta \\ \Delta \end{array}\right) e =
M_{\bar z} \left(\begin{array}{c} \theta -\theta(0) \\ \Delta \end{array}\right) e, 
\end{align*}
so
\begin{align*}
\cM_\theta P\ci{\cK_\theta} M_{\bar z} \left(\begin{array}{c} \theta \\ \Delta \end{array}\right) e
=
P\ci{\cK_\theta} \left(\begin{array}{c} \theta -\theta(0) \\ \Delta \end{array}\right) e
=
-P\ci{\cK_\theta} \left(\begin{array}{c} \theta(0) \\ 0 \end{array}\right) e.
\end{align*}
Combining this with \eqref{P_theta_N-F}, we get
\begin{align*}
\cM_\theta P\ci{\cK_\theta} M_{\bar z} \left(\begin{array}{c} \theta \\ \Delta \end{array}\right) e =
\left(\begin{array}{c}   \theta\theta(0)^* -\bI \\ \Delta\theta(0)^*\end{array}\right) \theta(0) e .
\end{align*}
Using the fact that 
\begin{align*}
\cM_\theta^* \left(\begin{array}{c} f \\ g \end{array}\right) = M_{\bar z}\left(\begin{array}{c} f - f(0)\\ g \end{array}\right), 
\end{align*}
we arrive at
\begin{align*}
\cM_\theta^*\cM_\theta P\ci{\cK_\theta} M_{\bar z} \left(\begin{array}{c} \theta \\ \Delta \end{array}\right) e = M_{\bar z} \left(\begin{array}{c} \theta -\theta(0)\\ \Delta \end{array}\right) \theta(0)^*\theta(0) e, 
\end{align*}
so
\begin{align*}
(\bI - \cM_\theta^*\cM_\theta) P\ci{\cK_\theta} M_{\bar z} \left(\begin{array}{c} \theta \\ \Delta \end{array}\right)e = M_{\bar z} \left(\begin{array}{c} \theta -\theta(0)\\ \Delta \end{array}\right)(\bI -  \theta(0)^*\theta(0) )e.
\end{align*}

Using the same reasoning as in the above proof of \eqref{C*_N-F} we get that 
\begin{align}
\label{comm-defects-03}
\f\bigl( (\bI - \cM_\theta^*\cM_\theta)\bigm|_{\fD_{\cM_\theta}} \bigr)
P\ci{\cK_\theta}  & M_{\bar z} \left(\begin{array}{c} \theta \\ \Delta \end{array}\right)e 
\\
\notag
&= M_{\bar z} \left(\begin{array}{c} \theta -\theta(0)\\ \Delta \end{array}\right)\f\bigl( \bI -  \theta(0)^*\theta(0)  \bigr)e, 
\end{align}
first with $\f$ being a polynomial, and then any measurable function. 

Using \eqref{comm-defects-03} with $\f(x)=x^{-1/2}$ and taking \eqref{C_N-F-02} into account, we get \eqref{C_N-F}. 
\end{proof}

\subsection{An auxiliary lemma}
\label{s:bound-C}

We already used, and we will also need later the following simple Lemma. 

\begin{lem}
\label{l:bound-C}
Let $\cM=\cM_\theta$ be model operator on a model space $\cK_\theta\subset L^2(W;\fD_*\oplus \fD)$, and let $\bC: \fD_*\to \fD\ci{\cM_\theta}$, $\bC_*: \fD\to \fD\ci{\cM_\theta^*}$ be bounded operators.

If  $C$ and $C_*$ are the operator-valued functions, defined by 
\begin{align*}
C(z) e & = \bC e (z),  &&z \in \T, \ e\in \fD , \\
C_*(z) e_* & = \bC_* e_* (z), &&z \in \T, \ e_*\in \fD_*. 
\end{align*}
then the functions $W^{1/2} C$ and $W^{1/2}C^*$ are bounded, 
\begin{align*}
\| W^{1/2} C \|\ci{L^\infty} = \|\bC\|, \qquad \| W^{1/2} C_* \|\ci{L^\infty} = \|\bC_*\| . 
\end{align*}		
\end{lem}

\begin{proof}
It is well-known and is not hard to show, that if $T$ is a contraction and $\cU $ is its unitary dilation, then then the subspaces $\cU^n \fD\ci T$, $n\in\Z$ (where recall $\fD\ci T$ is the defect space of $T$) are mutually orthogonal, and similarly for subspaces $\cU^n\fD\ci{T^*}$, $n\in\Z$. 	

Therefore, the subspaces $z^n \fD\ci{\cM}$, $n\in\Z$ are mutually orthogonal in $L^2(W; \fD_*\oplus \fD)$. and the same holds for the subspaces $z^n \fD\ci{\cM^*}$, $n\in\Z$. 

The subspaces $z^n \fD\subset L^2(\T;\fD)$ are mutually orthogonal, and since 
\[
C(z) \sum_{n\in\Z} z^n \hat f(n) = \sum_{n\in\Z} z^n \bC f_n , \qquad \hat f(n)\in \fD, 
\]
we conclude that the operator $f\mapsto Cf$ is a bounded operator acting $L^2(\fD)\to L^2(W;\fD_*\oplus\fD)$, and its norm is exactly $\|\bC\|$. 

But that means the multiplication  operator $f\mapsto W^{1/2}f$ between the non-weighted spaces $L^2(\fD)\to L^2(\fD_*\oplus\fD)$ is bounded with the same norm, which immediately implies that   $\| W^{1/2} C \|\ci{L^\infty} = \|\bC\|$.

The proof for $C_*$ follows similarly.  	
\end{proof}

\section{Characteristic function}\label{s-charfunc}


In this section we derive  formulas for the (matrix-valued) 
characteristic function $\theta\ci\Gamma$, see Theorem \ref{t-theta} below. 

\subsection{An inverse of a perturbation}
We begin with an auxiliary result.
\begin{lem}
\label{l:pert-01}
Let $D$ be an operator in an auxiliary Hilbert space ${\er}$ and let $B, C:{\er}\to \cH$. Then $\bI\ci\cH - CDB^*$ is invertible if and only if $\bI\ci {\er} - DB^*C$ is invertible, and if and only if $\bI\ci {\er} -B^*CD$ is invertible. 

Moreover, in this case
\begin{align}
\label{InvPert-01}
(\bI\ci\cH - CDB^*)^{-1} & = \bI\ci \cH + C (\bI\ci {\er} - DB^*C)^{-1} DB^*\\ \notag
                       & = \bI\ci \cH + C D (\bI\ci {\er} - B^*CD)^{-1}  B^*.
\end{align}
\end{lem}

We will apply this lemma for $D:\C^d\to \C^d$, so in this case the inversion of $\bI\ci\cH - CDB$ is reduced to inverting $(d\times d)$ matrix. 

This lemma can be obtained from the Woodbury inversion formula \cite{Wood}, although formally in \cite{Wood} only the matrix case was treated. 

\begin{proof}[Proof of Lemma \ref{l:pert-01}]
First let us note that it is sufficient to prove lemma with $D=\I\ci {\er}$, because $D$ can be incorporated either into $C$ or into $B^*$.  

One could guess the formula by writing the power series expansion of $\bI\ci\cH - CDB^*$, and we can get the result for the case when the series converges. This method can be made rigorous for finite rank  perturbations by considering the family $(\bI\ci\cH -\la CDB^*)^{-1}$, $\la\in\C$ and using  analytic continuation. 

However, the simplest way to prove the formula is just by performing multiplication, 
\begin{align*}
(\I\ci\cH - CB^*) \Bigl(\I\ci\cH & + C(\I\ci {\er} -B^*C)^{-1}B^*\Bigr)\\
& = \I\ci\cH -CB^* + C(\I\ci {\er} -B^*C)^{-1}B^* -CB^*C(\I\ci {\er} -B^*C)^{-1}B^* \\
&= \I\ci\cH  + C\Bigl( -\I\ci {\er} (\I\ci {\er} - B^*C) +\I\ci {\er} - B^*C \Bigr) (\I\ci {\er} - B^*C)^{-1}B^*  \\
&=\I\ci\cH. 
\end{align*}
Thus, when $\I\ci {\er} -B^*C$ is invertible, the operator $\I\ci\cH  + C(\I\ci {\er} -B^*C)^{-1}B^*$ is the right inverse of $\I\ci\cH - CB^*$. To prove that it is also a right inverse we even do not need to perform the multiplication: we can just take the adjoint of the above identity and then interchange $B$ and $C$. 

So, the invertibility of $\I\ci {\er} -B^*C$ implies the invertibility of $\I\ci\cH - CB^*$ and the formula for the inverse. To prove the ``if and only if'' statement we just need to change the roles of $\cH$ and ${\er}$ and express, using the just proved formula,  the inverse of $\I\ci {\er} -B^*C$ in terms of $(\I\ci\cH - CB^*)^{-1}$.
%
%
%
%
%
%
%
%
%
%
\end{proof}

\subsection{Computation of the characteristic function}
We turn to computing the characteristic function of $T=  U+ \bB(\Gamma-\bI\ci{\C^d}) \bB^*U$, $\|\Gamma\|<1$, where $U$ is the multiplication operator $M_\xi$ in $L^2(\mu;E)$.

We will use formula \eqref{CharFunction-01} with $V= {\bB}^*U$, $V_*={\bB}^*$, $\fD=\fD_*=\C^d$.

Let us first calculate for $|z|<1$:
\begin{align*}
(\bI\ci\cH - z T^*)^{-1} &= \left[ (\bI\ci\cH - zU^*) \left(\bI\ci\cH - z (\bI\ci\cH - zU^*)^{-1} U^*{\bB}(\Gamma^* -\bI\ci {\C^d}){\bB}^* \right)\right]^{-1} \\
& = \left[\bI\ci\cH - z (\bI\ci\cH - zU^*)^{-1} U^*{\bB}(\Gamma^* -\bI\ci {\C^d}){\bB}^* \right]^{-1} (\bI\ci\cH - zU^*)^{-1} \\
&=: X(z) (\bI\ci\cH - zU^*)^{-1}. 
\end{align*}
To compute the inverse  $X(z)$ we use Lemma \ref{l:pert-01} with $z (\bI\ci\cH - zU^*)^{-1} U^*{\bB}$ instead of $C$, $\Gamma^* -\bI\ci {\C^d}$ instead of $D$ and $\bB$ instead of $B$. Together with the first identity in \eqref{InvPert-01} we get 
\begin{align}
\label{X_1}
X(z)  = 
\bI\ci \cH + z(\bI\ci\cH -z U^*)^{-1} U^*{\bB} \Bigl(\bI\ci {\C^d} - z 
D {\bB}^*(\bI\ci \cH - z U^*)^{-1}U^*{\bB}\Bigr)^{-1} D 
{\bB}^*, 
\end{align}
where $D=\Gamma^*-\I\ci{\C^d}$. 

Now, let us express $z{\bB}^*(\bI\ci \cH - z U^*)^{-1}U^*{\bB}$ as a Cauchy integral of some matrix-valued measure. Recall that $U$ is a multiplication by the independent variable $\xi$ in $\cH\subset L^2(\mu;E)$. Recall that $b_1, b_2, \ldots , b_d\in \cH$ denote the ``columns'' of ${\bB}$ (i.e.~$b_k = {\bB}e_k$, where $e_1, e_2, \ldots, e_d$ is the standard basis in $\C^d$), and $B(\xi)= (b_1(\xi), b_2(\xi),\ldots, b_d(\xi))$ is the matrix with columns $b_k(\xi)$. Then
\begin{align*}
b_j^* (\bI\ci{\C^d} - zU^*)^{-1} U^* b_k = \int_\T \frac{\overline\xi}{1-z\overline\xi}\, {b_j(\xi)^*} b_k(\xi) \dd\mu(\xi), 
\end{align*}
so
\begin{align}
\label{MatrCauchy-01}
z {\bB}^*(\bI\ci \cH - z U^*)^{-1}U^*{\bB} = \int_\T \frac{z\overline\xi}{1-z\overline\xi} \,M(\xi)\dd\mu(\xi) =:  \cC_1[M\mu](z)=: F_1(z). 
\end{align}
where $M $ is the matrix-valued function $M(\xi)=B(\xi)^*B(\xi)$, or equivalently $M_{j,k}(\xi) = {b_j(\xi)^*} b_k(\xi)$, $1\le j, k \le d$.

Using \eqref{MatrCauchy-01}  and denoting $D:= \Gamma^* -\bI\ci{\C^d}$ we get from the above calculations that 
\begin{align*}
(\bI\ci\cH -  z T^*)^{-1}  =  \, & (\bI\ci\cH    - zU^*)^{-1} \\
& + z (\I\ci\cH - zU^*)^{-1} U^*{\bB} \Bigl(\I\ci{\C^d} -  D F_1(z) \Bigr)^{-1} D{\bB}^* (\bI\ci\cH - zU^*)^{-1} . 
\end{align*}

Applying formula \eqref{CharFunction-01}, with $V= {\bB}^*U$, $V_*={\bB}_*$, $\fD=\fD_*=\C^d$, we see that the characteristic function is an analytic function $\theta=\theta\ci T$, whose values are bounded linear operators acting on $\fD$,  defined by the 
formula 
\begin{align}
\label{CharFunction}
\theta_T(z) = {\bB}^*\left( -T + z D_{T^*}\left(\I\ci\cH - z T^*\right)^{-1} D_T \right) U^*{\bB} \Bigm|_{\fD}, \qquad z\in \D. 
\end{align}

We can  see from \eqref{BlDec-01} that the defect operators $D\ci T$ and $D\ci{T^*}$ are given by 
\begin{align*}
D\ci T= U^*{\bB} D\ci\Gamma {\bB}^*U, \qquad D\ci{T^*} = {\bB} D\ci{\Gamma^*} {\bB}^*.
\end{align*}

We can also see from \eqref{BlDec-01} that the term $-T$ in \eqref{CharFunction} contributes $-\Gamma$ to the matrix 
$\theta\ci T$. The rest can be obtained from the above representation formula for $(\I\ci\cH -zT^*)^{-1}$. 
Thus, recalling the definition \eqref{MatrCauchy-01} of $\cC_1M\mu$ 
we get, denoting $F_1(z):= (\cC_1M\mu)(z)$, that
\begin{align*}
\theta\ci T (z) & = -\Gamma +  D\ci{\Gamma^*}\Biggl[ F_1(z)  +  F_1(z) \Bigl( \I\ci{\fD} -(\Gamma^*-\I\ci{\fD})F_1(z)\Bigr)^{-1} (\Gamma^*-\I\ci{\fD}) F_1(z) \Biggr] D\ci\Gamma
\\
& = -\Gamma  +   D\ci{\Gamma^*} F_1(z) \Bigl( \I\ci{\fD} -(\Gamma^*-\I\ci{\fD})F_1(z)\Bigr)^{-1}  D\ci\Gamma. 
\end{align*}

In the above computation to compute $X(z)$ we can use the second formula in \eqref{InvPert-01}. We get instead of \eqref{X_1} an alternative representation
\begin{align*}
X(z)  = 
\bI\ci \cH + z(\bI\ci\cH -z U^*)^{-1} U^*{\bB} D \Bigl(\bI\ci{\fD} - z 
 {\bB}^*(\bI\ci \cH - z U^*)^{-1}U^*{\bB}D\Bigr)^{-1}   
{\bB}^*. 
\end{align*}
Repeating the same computations as above we get another formula for $\theta\ci T$, 
\begin{align*}
\theta\ci T (z) = -\Gamma  +   D\ci{\Gamma^*}  \Bigl( \I\ci{\fD} -F_1(z)(\Gamma^*-\I\ci{\fD})\Bigr)^{-1} F_1(z)  D\ci\Gamma.
\end{align*}

To summarize we have proved two representations of the 
characteristic operator-valued function.
\begin{theo}
\label{t-theta}
Let $T=T_\Gamma$ be the operator given in \eqref{BlDec-01}, with $\Gamma$ being a strict contraction. Then the characteristic function $\theta\ci{T}=\theta\ci{T_\Gamma}\in H^\infty(\fD\,\tto\fD_*)$,  with coordinate operators $V= {\bB}^*U$, $V_*={\bB}^*$ (and with  $\fD=\fD_*=\C^d$) is given by
\begin{align*}
\theta\ci{T_\Gamma} (z)  &= -\Gamma  +   D\ci{\Gamma^*} F_1(z) \Bigl( \I\ci{\fD} -(\Gamma^*-\I\ci{\fD})F_1(z)\Bigr)^{-1}  D\ci\Gamma
\\
&= -\Gamma  +   D\ci{\Gamma^*}  \Bigl( \I\ci{\fD} -F_1(z)(\Gamma^*-\I\ci{\fD})\Bigr)^{-1} F_1(z)  D\ci\Gamma,
\end{align*}
where $F_1(z)$ is the matrix-valued function given by \eqref{MatrCauchy-01}.
\end{theo}

In these formulas, the inverse is taken of a $(d\times d)$ matrix-valued function, which is much simpler than computing the inverse in \eqref{CharFunction}.

\subsection{
Characteristic function and the Cauchy integrals of matrix-valued measures}

For a (possibly complex-valued) measure $\tau$ on $\T$ and $z\notin\T$ define the following Cauchy type transforms $\cC$, $\cC_1$ and $\cC_2$
\begin{align*}
\cC \tau (z) := \int_\T \frac{d\tau(\xi)}{1-\overline\xi z}, \qquad \cC_1 \tau ( z) := \int_\T \frac{\overline\xi  z d\tau(\xi)}{1-\overline\xi z}, \qquad \cC_2\tau( z):= \int_\T \frac{1+ \overline\xi  z }{1-\overline\xi z} d\tau(\xi).
\end{align*}

Performing the Cauchy transforms component-wise we can define them for matrix-valued measures as well. 

Thus $F_1$ from the above Theorem \ref{t-theta} is given by $F_1 =\cC_1[M\mu]$, where $M(\xi) = B^*(\xi) B(\xi)$. We would like to give the representation of $\theta\ci{T_\Gamma}$ in terms of function  $F_2:= \cC_2 [M\mu]$.


Slightly abusing notation we will write $\theta_\Gamma$ instead of 
$\theta\ci{T_\Gamma}$.

\begin{cor}
\label{c:theta_0}
For $\theta\ci\OZ :=\theta\ci{T_\OZ}$ we have
\begin{align}
\label{theta_0-01}
 \theta\ci\OZ(z) & = F_1(z) (\bI + F_1(z))^{-1} = (\bI + F_1(z))^{-1}F_1(z) \\
\label{theta_0-02}
  & = (F_2(z)-\bI) (F_2(z)+\bI)^{-1} = (F_2(z)+\bI)^{-1}  (F_2(z)-\bI) .
\end{align}
\end{cor}
\begin{proof}
The identity \eqref{theta_0-01} is a direct application of Theorem \ref{t-theta}. The identity \eqref{theta_0-02} follows immediately from the trivial relation
\[
F_2(z) = \int_\T M\dd \mu + 2F_1(z) = \bI\ci\fD + 2 F_1(z);
\]
the equality $\int_\T M\dd \mu =\bI\ci\fD=\bI\ci{\C^d}$ is just a re-statement of the fact that the functions $b_1, b_2, \ldots, b_d$ form an orthonormal basis in $\cH$. 
\end{proof}

\section{Relations between characteristic functions \texorpdfstring{$\theta\ci\Gamma$}{theta<sub>Gamma} }
\label{s:PropCharFunct}

\subsection{Characteristic functions and linear fractional transformations}\label{ss:LFT}

When $d=1$, it is known that the characteristic functions are related by a linear fractional transformation
\[
 \te_\gamma(z) = \frac{\te_0(z)-\gamma}{1-\overline\gamma \te_0(z)}\,,
 \]
see \cite[Equation (2.9)]{LT15}. 

It turns out that a similar formula holds for finite rank perturbations. 

\begin{theo}
\label{t-LFT}
Let $T$ be the operator given in \eqref{BlDec-01}, with $\Gamma$ being  a strict contraction. Then the characteristic functions $\theta\ci{\Gamma}:=\theta\ci{T_\Gamma}$ and $\te\ci\OZ = \theta\ci{T_\OZ} $ are related via linear fractional transformation
\[
\theta\ci{\Gamma}
=
D\ci{\Gamma^*}^{-1}(  \theta\ci\OZ-\Gamma) (\OID\ci{\fD} - \Gamma^*  \theta\ci\OZ)^{-1} D\ci{\Gamma}
=
D\ci{\Gamma^*}(\OID\ci{\fD} -  \theta\ci\OZ\Gamma^* )^{-1} (  \theta\ci\OZ-\Gamma) D\ci{\Gamma}^{-1}
.
\]
\end{theo}

\begin{rem*}
At first sight, this formula looks like a formula in \cite[p.~234]{Nik-Vas_model_MSRI_1998}. However, their result expresses the characteristic function in terms of a linear fractional transformation in $T$; whereas, here we have a linear fractional transformation in $\Gamma$.
\end{rem*}

\begin{theo}
\label{t:LFT-02}
Under assumptions of the above Theorem \ref{t-LFT}
\[
\theta_\OZ = D\ci{\Gamma^*} (\bI + \theta\ci\Gamma  \Gamma^*)^{-1} (\theta\ci\Gamma + \Gamma) D\ci\Gamma^{-1} 
=  D\ci{\Gamma^*}^{-1}  (\theta\ci\Gamma + \Gamma) (\bI + \Gamma^* \theta\ci\Gamma)^{-1} D\ci\Gamma.
\]
\end{theo}

To prove Theorem \ref{t-LFT} we start with the following simpler statement. 
\begin{prop}\label{p-LFT}
The matrix-valued characteristic functions $\theta\ci{\Gamma}$ and $\te\ci\OZ$ are related via
\[
\theta\ci{\Gamma}
=
-\Gamma + D\ci{\Gamma^*}  \theta\ci\OZ \left(\OID\ci{\fD} - \Gamma^*  \theta\ci\OZ\right)^{-1} D\ci{\Gamma}
=
-\Gamma + D\ci{\Gamma^*}  \left(\OID\ci{\fD} -   \theta\ci\OZ\Gamma^*\right)^{-1}  \theta\ci\OZ D\ci{\Gamma}
.
\]
\end{prop}

\begin{proof}

Solving \eqref{theta_0-01}   for $F_1$ we get that
\[
F_1(z) = \theta\ci\OZ(z)[\I-\theta\ci\OZ(z)]^{-1}.
\]
Substituting this expression into the formula for the characteristic function from Theorem \ref{t-theta}, we see that
\begin{align}
\label{theta_Gamma-03}
\theta\ci\Gamma   = -\Gamma  +   D\ci{\Gamma^*}  \theta\ci\OZ[\I\ci{\fD}- \theta\ci\OZ]^{-1} \Bigl\{ \I\ci{\fD} -(\Gamma^*-\I\ci{\fD}) \theta\ci\OZ[\I\ci{\fD}- \theta\ci\OZ]^{-1}\Bigr\}^{-1}  D\ci\Gamma.
\end{align}

We manipulate the term inside the curly brackets
\begin{align*}
\I\ci{\fD} -(\Gamma^*-\I\ci{\fD}) \theta\ci\OZ
[\I\ci{\fD}- \theta\ci\OZ]^{-1}
&=
\left(
\I\ci{\fD} -  \theta\ci\OZ -(\Gamma^*-\I\ci{\fD}) \theta\ci\OZ
\right)
[\I\ci{\fD}- \theta\ci\OZ]^{-1}
\\
&=
\left(
\I\ci{\fD} -\Gamma^*  \theta\ci\OZ
\right)
[\I\ci{\fD}- \theta\ci\OZ]^{-1},
\end{align*}
so that 
\[
 \Bigl\{ \I\ci{\fD} -(\Gamma^*-\I\ci{\fD}) \theta\ci\OZ[\I\ci{\fD}- \theta\ci\OZ]^{-1}\Bigr\}^{-1} 
 =
 [\I\ci{\fD}- \theta\ci\OZ]
 \left(
\I\ci{\fD} -\Gamma^*  \theta\ci\OZ
\right)
^{-1}.
\]
Substituting this back into \eqref{theta_Gamma-03}, we get the first equation the first equation in the proposition.

The second equation is obtained similarly.
\end{proof}

\begin{lem}
\label{l:G-D_G}
For $\|\Gamma\|<1$ we have for all $\alpha\in\R$
\begin{align}
\label{G-D_G-01}
D\ci{\Gamma^*}^{\alpha}\Gamma & = \Gamma  D\ci{\Gamma}^{\alpha} \,,\\
\label{G-D_G-02}
D\ci{\Gamma}^{\alpha}\Gamma^* & = \Gamma^*  D\ci{\Gamma^*}^{\alpha}\,, 
\end{align}
where, recall $D\ci \Gamma := (\bI - \Gamma^*\Gamma)^{1/2}$, $D\ci{\Gamma^*} := (\bI - \Gamma\Gamma^*)^{1/2}$ are the defect operators. 
\end{lem}
\begin{proof}
Let us prove \eqref{G-D_G-01}. It is trivially true for $\alpha=2$, and by induction we get that it is true for $\alpha= 2n$, $n\in\N$. Since $\|\Gamma\|<1$, the spectrum of $D\ci\Gamma$ lies in the interval $[ a, 1]$, $a=(1-\|\Gamma\|^2)^{1/2}>0$. 

Approximating $\f(x)= x^\alpha$ uniformly on $[a,1]$ by polynomials of $x^2$ we get \eqref{G-D_G-01}. 

Applying \eqref{G-D_G-01} to $\Gamma^*$ we get \eqref{G-D_G-02}. 
\end{proof}

%
%
%
%

\begin{proof}[Proof of Theorem \ref{t-LFT}]
From \eqref{G-D_G-01} we get that $D\ci{\Gamma^*}^{-1}\Gamma D\ci{\Gamma}^{-1} = D\ci{\Gamma^*}^{-2}\Gamma$, so 
\begin{align*}
\theta\ci{\Gamma}
&=
-\Gamma + D\ci{\Gamma^*}  \theta\ci\OZ \left(\OID\ci{\fD} - \Gamma^*  \theta\ci\OZ\right)^{-1} D\ci{\Gamma}
\\&
=
D\ci{\Gamma^*}\left[-D\ci{\Gamma^*}^{-2}\Gamma +  \theta\ci\OZ \left(\OID\ci{\fD} - \Gamma^*  \theta\ci\OZ\right)^{-1} \right]
D\ci{\Gamma}
\\&
=
D\ci{\Gamma^*}^{-1}\left[-\Gamma + D\ci{\Gamma^*}^{2} \theta\ci\OZ \left(\OID\ci{\fD} - \Gamma^*  \theta\ci\OZ\right)^{-1} \right]
D\ci{\Gamma}
\\&
=
D\ci{\Gamma^*}^{-1}\left[-\Gamma (\OID\ci{\fD} - \Gamma^*  \theta\ci\OZ)+(\bI -\Gamma\Gamma^*)  \theta\ci\OZ  \right]
\left(\OID\ci{\fD} - \Gamma^*  \theta\ci\OZ\right)^{-1}
D\ci{\Gamma}
\\&
=
D\ci{\Gamma^*}^{-1}\left[-\Gamma  +  \theta\ci\OZ  \right]
\left(\OID\ci{\fD} - \Gamma^*  \theta\ci\OZ\right)^{-1}
D\ci{\Gamma},
\end{align*}
which is exactly the first identity. 

The second identity is obtained similarly, using the formula $D\ci{\Gamma^*}^{-1}\Gamma D\ci{\Gamma}^{-1} = \Gamma D\ci{\Gamma}^{-2}$ and taking the factor $\left(\OID\ci{\fD} - \Gamma^*  \theta\ci\OZ\right)^{-1}$ out of brackets on the left. 
\end{proof}

\begin{proof}[Proof of Theorem \ref{t:LFT-02}]
Right multiplying the first identity in Theorem \ref{t-LFT} by $D\ci\Gamma^{-1} (\bI-\Gamma^* \theta\ci\OZ)$ we get 
\[
\theta\ci\Gamma D\ci\Gamma^{-1} - \theta\ci\Gamma D\ci\Gamma^{-1} \Gamma^* \theta\ci\OZ = D\ci{\Gamma^*}^{-1} \theta\ci\OZ - D\ci{\Gamma^*}^{-1} \Gamma. 
\]
Using identities $D\ci{\Gamma^*}^{-1}\Gamma =\Gamma D\ci\Gamma^{-1}$ and $D\ci\Gamma^{-1} \Gamma^* = \Gamma^* D\ci{\Gamma^*}^{-1}$, see Lemma \ref{l:G-D_G}, we rewrite the above equality as 
\[
\theta\ci\Gamma D\ci\Gamma^{-1} + \Gamma D\ci\Gamma^{-1} = 
\theta\ci\Gamma  \Gamma^* D\ci{\Gamma^*}^{-1}  \theta\ci\OZ + D\ci{\Gamma^*}^{-1} \theta\ci\OZ.
\]
Right multiplying both sides by $D\ci{\Gamma^*}(\theta\ci\Gamma \Gamma^* + \bI)^{-1} $ we get the first equality in the theorem. 

The second one is proved similarly. 
\end{proof}


\subsection{The defect functions \texorpdfstring{$\Delta\ci\Gamma$}{Delta<sub>Gamma} and relations between them}\label{s-Delta}
Recall that every strict contraction $\Gamma$ yields a characteristic matrix-valued function $\te\ci\Gamma$ through association with the c.n.u.~contraction $U\ci\Gamma$. The definition of the Sz.-Nagy--Foia\c s model space (see e.g.~formula \eqref{N-F-K_theta}) reveals immediately that the defect functions $\Delta\ci\Gamma= (\OID - \te^*\ci\Gamma \te\ci\Gamma)^{1/2}$ are central objects in model theory. We express defect function $\Delta\ci\Gamma$ in terms of $\Delta\ci\OZ$ (and $\Gamma$ and $\te\ci\OZ$).

\begin{theo}\label{t-Delta}
The defect functions of $\te\ci\Gamma$ and $\te\ci\OZ$ are related by
\[
\Delta\ci\Gamma^2
=
D\ci\Gamma (I-\te\ci\OZ^*\Gamma)^{-1} \Delta\ci\OZ^2(I-\Gamma^* \te\ci\OZ)^{-1}
D\ci\Gamma .
\]
\end{theo}

\begin{proof}
By Theorem \ref{t-LFT} 
\begin{align*}
\theta\ci{\Gamma}
=
D\ci{\Gamma^*}^{-1}(  \theta\ci\OZ-\Gamma) (\OID\ci{\fD} - \Gamma^*  \theta\ci\OZ)^{-1} D\ci{\Gamma}, 
\end{align*}
so $\theta\ci\Gamma^*\theta\ci\Gamma = A^*B A$, where 
\begin{align*}
A=(\bI -\Gamma^* \theta\ci\OZ) D\ci\Gamma, \qquad B = (\theta\ci\OZ^*-\Gamma^*)  
D\ci{\Gamma*}^{-2} (\theta\ci\OZ-\Gamma)  .
\end{align*}
Then
$
\Delta\ci\Gamma = \bI - \theta\ci\Gamma^* \theta\ci\Gamma = A^*XA$, 
where 
\begin{align*}
X &=(A^*)^{-1} A^{-1} - B
= (\bI -\theta\ci\OZ^*\Gamma ) D\ci\Gamma^{-2} (\bI -\Gamma^* \theta\ci\OZ) - 
(\theta\ci\OZ^*-\Gamma^*)  D\ci{\Gamma*}^{-2} (\theta\ci\OZ-\Gamma) \\
& = 
D\ci\Gamma^{-2} - \theta\ci\OZ^*\Gamma D\ci\Gamma^{-2} - D\ci\Gamma^{-2}\Gamma^* \theta\ci\OZ +
\theta\ci\OZ^*\Gamma D\ci\Gamma^{-2} \Gamma^* \theta\ci\OZ
\\ &
\qquad\qquad\qquad  -\theta\ci\OZ^* D\ci{\Gamma^*}^{-2} \theta\ci\OZ + \Gamma^* D\ci{\Gamma^*}^{-2} \theta\ci\OZ + 
\theta\ci\OZ^* D\ci{\Gamma^*}^{-2} \Gamma -\Gamma^* D\ci{\Gamma^*}^{-2} \Gamma
\end{align*}
It follows from Lemma \ref{l:G-D_G} that $ D\ci\Gamma^{-2} \Gamma^* = \Gamma^*D\ci{\Gamma^*}^{-2}$ and that $\Gamma^* D\ci\Gamma^{-2} = D\ci{\Gamma^*}^{-2} \Gamma$, so in the above identity we have cancellation of non-symmetric terms, 
\begin{align*}
- \theta\ci\OZ^*\Gamma D\ci\Gamma^{-2} - D\ci\Gamma^{-2}\Gamma^* \theta\ci\OZ 
+ \Gamma^* D\ci{\Gamma^*}^{-2} \theta\ci\OZ + 
\theta\ci\OZ^* D\ci{\Gamma^*}^{-2} \Gamma  = 0. 
\end{align*}
Therefore
\begin{align*}
X & = D\ci\Gamma^{-2} + \theta\ci\OZ^*\Gamma D\ci\Gamma^{-2} \Gamma^* \theta\ci\OZ
-\theta\ci\OZ^* D\ci{\Gamma^*}^{-2} \theta\ci\OZ  -\Gamma^* D\ci{\Gamma^*}^{-2} \Gamma
\\ 
& = D\ci\Gamma^{-2} + \theta\ci\OZ^* D\ci{\Gamma^*}^{-2} \Gamma\Gamma^* \theta\ci\OZ
-\theta\ci\OZ^* D\ci{\Gamma^*}^{-2} \theta\ci\OZ  - D\ci{\Gamma}^{-2} \Gamma^* \Gamma 
\\ & 
= D\ci\Gamma^{-2} ( \bI - \Gamma^*\Gamma) + \theta\ci\OZ^* D\ci\Gamma^{-2}( \Gamma^*\Gamma -\bI) \theta\ci\OZ =\bI - \theta\ci\OZ^*\theta\ci\OZ =\Delta\ci\OZ .
\end{align*}
Thus we get that $\Delta\ci\Gamma = A^* \Delta\ci\OZ A$, which is exactly the conclusion of the theorem. 
\end{proof}

\subsection{Multiplicity of the absolutely continuous spectrum}
It is well-known that the Sz.-Nagy--Foia\c s model space reduces to the familiar one-story setting with $\mathcal{K}_\te = H^2(\fD_*)\ominus \te H^2(\fD)$ when $\te$ is inner. Indeed, for inner $\te$ the non-tangential boundary values of the defect $\Delta(\xi) = (\OID - \te^*(\xi) \te(\xi))^{1/2}= 0$ Lebesgue a.e.~$\xi\in \T$. So, the second component of the Sz.-Nagy--Foia\c s model space collapses completely.

Here we provide a finer result that reveals the matrix-valued weight function and the multiplicity of $U$'s absolutely continuous part.

Before we formulate the statement, we recall some terminology. First, we Lebesgue decompose the (scalar) measure $d\mu = d\mu\ti{ac}+d\mu\ti{sing}$. The absolutely continuous part of $U$ is unitarily equivalent to the multiplication by the independent variable $\xi$ on the von Neumann direct integral $\cH\ti{ac}= \int_\T^\oplus E(\xi) \dd \mu\ti{ac}(\xi).$ Note that the dimension of $E(\xi)$ is the multiplicity function of the spectrum.

Let $w$ denote the density of the absolutely continuous part of $\mu$, i.e.~$\dd\mu\ti{ac}(\xi) = w(\xi)\dd m(\xi)$. Then the matrix-valued function $\xi\mapsto B^*(\xi)B(\xi)w(\xi)$ is the absolutely continuous part of the matrix-valued measure
$
B^*B\mu
$.

\begin{theo}\label{t-AC}
The defect function $\Delta\ci\OZ$ of $\theta\ci\OZ$ and the absolutely continuous part $B^*Bw$ of the matrix-valued  measure $B^*B\mu$ are related by 
\begin{align}
\label{e-NOT}
(\OID - \te\ci\OZ^*(\xi))B^*(\xi)B(\xi)w(\xi)
(\OID - \te\ci\OZ(\xi))=
(\Delta\ci\OZ(\xi))^2
\end{align}
for Lebesgue a.e.~$\xi\in \T$.

The function $\bI-\theta\ci\OZ$ is invertible a.e.~on $\T$, so
 the multiplicity of the absolutely continuous part of $\mu$ is given by
\begin{align}\label{e-3}
\dim E(\xi)
=
\rank (\OID - \te\ci\OZ^*(\xi)\te\ci\OZ(\xi))
=
\rank\bigtriangleup\ci\OZ(\xi),
\end{align}
of course, with respect to Lebesgue a.e.~$\xi\in \T$.
\end{theo}

Combining \eqref{e-3} with Theorem \ref{t-Delta} we obtain:

\begin{cor}
For Lebesgue a.e.~$\xi\in\T$ we have
$\dim E(\xi)
=
\rank\bigtriangleup\ci\Gamma(\xi)
$ for all strict contractions $\Gamma$.
\end{cor}

Another immediate consequence is the following:
\begin{cor}
Operator $U$ has no absolutely continuous part on a Borel set $B\subset \T$ if and only if $\te\ci\OZ(\xi)$ (or, equivalently, $\te\ci\Gamma(\xi)$ for all strict contractions $\Gamma$) is unitary for Lebesgue almost every $\xi\in B$.
\end{cor}
This corollary is closely related to the main result of \cite[Theorem 3.1]{DL2013}. Interestingly, it appears that the proof (in \cite{DL2013}) of that result cannot be refined to yield our current result (Theorem \ref{t-AC}).

\begin{cor}\label{c-TFAE}
In particular, we confirm that the following are equivalent:
\begin{enumerate}
\item $U$ is purely singular,
\item $\te\ci\Gamma(\xi)$ is inner for one (equivalently any) strict contraction $\Gamma$,
\item $\Delta\ci\Gamma \equiv \OZ$ for one (equivalently any) strict contraction $\Gamma$,
\item the second story of the Sz.-Nagy--Foia\c s model space collapses (and we are dealing with the model space $\cK_{\te\ci\Gamma} = H^2(\C^d) \ominus \te\ci\Gamma H^2(\C^d)$ for one (equivalently any) strict contraction $\Gamma$).
\end{enumerate}
\end{cor}

\begin{proof}[Proof of Theorem \ref{t-AC}]
Take $\Gamma\equiv\OZ$. Solving \eqref{theta_0-02} for $F_2$ we see
\[
F_2(z) = [\OID + \te\ci\OZ(z)][\OID - \te\ci\OZ(z)]^{-1}
.
\]

Let $\cP(B^*B\mu)$ denote the Poisson extension of the matrix-valued measure $B^*B\mu$ to the unit disc $\D$. Since $F_2= \cC_2 B^*B\mu$, we can see that $\cP(B^*B\mu)= \re F_2$ on $\D$, so  
\[
\cP(B^*B\mu)=\re F_2
=
\re[(\OID+\te\ci\OZ)(\OID-\te\ci\OZ)^{-1}].
\]
Standard computations yield
\begin{align*}
\cP(B^*B\mu)
&=
\re[(\OID+\te\ci\OZ)(\OID-\te\ci\OZ)^{-1}]
=
\frac{1}{2} [(\OID+\te\ci\OZ)(\OID-\te\ci\OZ)^{-1}+(\OID-\te\ci\OZ^*)^{-1}(\OID+\te\ci\OZ^*)]\\
&=
\frac{1}{2} (\OID-\te\ci\OZ^*)^{-1}\left[(\OID-\te\ci\OZ^*)(\OID+\te\ci\OZ)+(\OID+\te\ci\OZ^*)(\OID-\te\ci\OZ)\right](\OID-\te\ci\OZ)^{-1}\\
&=
\frac{1}{2} (\OID-\te\ci\OZ^*)^{-1}[\OID - \te\ci\OZ^*\te\ci\OZ](\OID-\te\ci\OZ)^{-1}
=
(\OID-\te\ci\OZ^*)^{-1}\re [\OID - \te\ci\OZ^*\te\ci\OZ](\OID-\te\ci\OZ)^{-1}\\
&=(\OID-\te\ci\OZ^*)^{-1}[\OID - \te\ci\OZ^*\te\ci\OZ](\OID-\te\ci\OZ)^{-1}
\end{align*}
on $\D$. Note that for any characteristic function $\theta$ and $z\in\D$ the matrix $\theta(z)$ is a strict contraction, so in our case $\bI -\theta\ci\OZ$ is invertible on $\D$, and all computations are justified. 

We can rewrite the above identity as 
\begin{align*}
(\bI-\theta\ci\OZ)^*\cP(B^*B\mu) (\bI -\theta\ci\OZ) = \OID - \te\ci\OZ^*\te\ci\OZ, 
\end{align*}
and taking the non-tangential boundary values we get \eqref{e-NOT}. Here we used the Fatou Lemma (see e.g.~\cite[Theorem 3.11.7]{NikERI}) which says that for a complex measure $\tau$ the non-tangential boundary values of its Poisson extension $\cP\tau$ coincide a.e.~with the density of the absolutely continuous part of $\tau$; applying this lemma entrywise we get what we need in the left hand side.

To see that the boundary values of $\bI-\te\ci\OZ$ are invertible a.e.~on $\T$ we notice that $z\mapsto \det (\bI- \te\ci\OZ(z)) $ is a bounded analytic function on $\D$, so its boundary values are non-zero a.e.~on $\T$. 
%
%
%
%
%
\end{proof}

\section{What is wrong with the universal representation formula and what to do about it?}\label{s-Explanations}

There are several things that are not completely satisfactory with the universal representation formula given by Theorem \ref{t-repr}.

First of all, it is defined only on functions of form $hb$, where $h\in C^1$ is a scalar function and $b\in\Ran\bB$.  Of course, one can than define it on a dense set, for example on the dense set of linear combinations $f=\sum_k h_k, b_k$, where $b_k$ are columns of the matrix $B$, $b_k=\bB e_k$, and $h_k\in C^1(\T)$. But the use of functions $b$ (or $b_k$) in the representation is a bit bothersome, especially taking into account that the representation $f=\sum_k h_k b_k$ is not always unique. So, it would be a good idea to get rid of the function $b$. 

The second thing is that while the representation formula looks like a singular integral operator (Cauchy transform), it is not represented as a classical singular integral operator, so it is not especially clear if the (well developed) theory of such operators apply in our case. So, we would like to represent the operator in more classical way.

Denoting
$C_1(z):=C_*(z)- z C(z)$ and using  the formal Cauchy-type expression
\[
(T^{B^* \mu} f )(z)
=
\int\ci{\T}
\frac{1}{1-z\bar\xi}\, B^*(\xi) f(\xi) d\mu(\xi),
\]
we can, performing formal algebraic manipulations, rewrite \eqref{e-repr} as
\begin{align}\label{e-reprSzNFExplanations}
(\Phi^* h b)(z)
=
C_1(z)
(T^{B^* \mu} h b)(z)
+h(z)
[C_*(z){\bB}^*b
-C_1(z)
(T^{B^* \mu} b)(z)]
,\quad z\in \T.
\end{align}

So, is it possible to turn these formal manipulations into meaningful mathematics? And the answer is ``yes'': the formula \eqref{e-reprSzNFExplanations} gives the representation of $\Phi^*$ if one interprets $T^{B^*\mu}f$ as the boundary values of the Cauchy Transform $\cC[B^*f\mu](z)$, $z\notin \T$, see the definition in the next section. 

In the next section (Section \ref{s-SIO}) we present necessary facts about (vector-valued) Cauchy transform and its regularization, that will allow us to interpret and justify the formal expression \eqref{e-reprSzNFExplanations}.  We will complete this justification in Section \ref{ss-PhiStarSNF}, see \eqref{form-repr-01}. This representation is a universal one, meaning that it works in any transcription of the model, but still involves the function $b\in \Ran\bB$.

The function $b$ is kind of eliminated Proposition \ref{p:repr2} below, and as it is usually happens in the theory of singular integral operators, the operator $\Phi^*$ splits into the singular integral part (weighted boundary values of the Cauchy transform)  and the multiplication part. The function $b$ becomes hidden in the multiplication part, and at the first glance it is not clear why this part is well defined. 

Thus the representation given by Proposition \ref{p:repr2} is still not completely satisfactory (the price one pays for the universality), but it is a step  to obtain a nice representations for a fixed transcription of a model. Thus we were able to obtain  a precise and unambiguous representation of $\Phi^*$ in the Sz.-Nagy--Foia\c{s} transcription, see Theorem \ref{t-repr3} which is the main result of Section \ref{ss-PhiStarSNF}.

\section{Singular integral operators}\label{s-SIO}

\subsection{Cauchy type integrals}

For a finite (signed or even complex-valued) measure $\nu$ on $\T$  its Cauchy Transform $\cC\nu$ is defined as 
\begin{align*}
\cC\nu(z) =\cC[\nu](z) = \int_\T\frac{\dd\nu(\xi)}{1-\bar\xi z} \,, \qquad z\in \C\setminus \T. 
\end{align*}
It is a classical fact that $\cC\nu(z) $ has non-tangential boundary values as $z\to z_0\in\T$ from the inside and from the outside of the disc $\D$. So, given a finite positive Borel measure $\mu$ one can define operators $T_\pm^\mu$ from $L^1(\mu; E)$ to the space of measurable functions on $\T$ as the non-tangential boundary values from inside and outside of the unit disc $\D$, 
\[
(T_{+}^\mu f)(z_0)
=
\text{n.t.-}\lim_{\substack{z\to z_0\\z\in\D}}
\cC [f\mu](z)\,,
 \qquad
 \qquad
(T_{-}^\mu f)(z_0)
=
\text{n.t.-}\lim_{\substack{z\to z_0\\ z\notin \overline{\D}}}
\cC [f\mu](z)
\,.
\]

One can also define the regularized operators $T^\mu_r$, $r\in (0,\infty)\setminus \{1\}$, and the restriction of $\cC [f\mu]$ to the circle of radius $r$, 
\begin{align*}
T^\mu_r f(z) = \cC [f \mu ](rz). 
\end{align*}

Everything can be extended to the case of vector and matrix valued measures; there are some technical details that should be taken care of in the infinite dimensional case, but in our case everything is finite dimensional ($\dim E\le d<\infty$), so the generalization is pretty straightforward. 

So, given a (finite, positive) scalar measure $\mu$ and a matrix-valued function $B^*$ (with entries in $L^2(\mu)$) and vector-valued function $f\in L^2(\mu; E)$ we can define  $T^{B^*\mu}_\pm f$ and $T_r^{B^*\mu} f$ as the non-tangential boundary values and the restriction to the circle of radius $r$ respectively of the Cauchy transform $\cC[ B^*f\mu](z)$. Modulo slight abuse of notation this notation agrees with the accepted notation for the scalar case.  

In what follows the function $B^*$ will be the function $B^*$ from Theorem \ref{t-repr}. 

\subsection{Uniform boundedness of the boundary Cauchy operator and its regularization}

For a finite Borel measure $\nu$ on $\T$ and $n\in \Z$ define 
\begin{align*}
P_n \nu (z) = \left\{ 
\begin{array}{ll} \sum_{k=0}^n \hat \nu (k) z^k \qquad & n\ge 0, \\
\sum_{k=n}^{-1} \hat \nu (k) z^k \qquad & n< 0 ; \end{array} \right. 
\end{align*}
here $\hat\nu(k)$ is the Fourier coefficient of $\nu$, $\hat\nu(k) = \int_\T \xi^{-k} \dd\nu(\xi)$. 

Recall that $C_1(z) := C_*(z) - zC(z)$ where $C_*$ and $C$ are from Theorem \ref{t-repr}. 

Recall that if $W$ is a matrix-valued weight (i.e.~a function whose values $W(\xi)$ are positive semidefinite operators on a finite-dimensional space $H$), then the norm in the weighted space $L^2(W;H)$ is defined as 
\[
\|f\|\ci{L^2(W; H)}^2 = \int_\T (W(\xi) f(\xi) , f(\xi))\ci H \dd m(\xi).
\]

We are working with the model space $\cK_\theta$ which is a subspace of a weighted space $L^2(W;\fD_*\oplus\fD)$ (the weight could be trivial, $W\equiv \bI$, as in the case of Sz.-Nagy--Foia\c{s} model). 

Define $\wt C_1:= W^{1/2} C_1$. 
The function $\wt C_1^*\wt C_1$ is a matrix-valued weight, whose values are operators on $\fD_*\oplus \fD$, so we can define the weighted space $L^2(\wt C_1^*\wt C_1) = L^2(\wt C_1^*\wt C_1;\fD_*\oplus\fD)$. Note that 
\[
\|f\|\ci{L^2(\wt C_1^*\wt C_1)} = \|\wt C_1 f\|\ci{L^2(\fD_*\oplus\fD)} = \| C_1 f\|\ci{L^2(W;\fD_*\oplus\fD)} . 
\]

\begin{lem}
\label{l:BdPn}
The operators $P_n^{B^*\mu}:\cH\subset L^2(\mu;E) \to L^2(\wt C_1^*\wt C_1;\fD_*\oplus\fD)$ defined by 
\begin{align*}
P_n^{B^*\mu} f := P_n (B^*f\mu), \qquad n\in \Z
\end{align*}
are uniformly in $n$ bounded with norm at most $2$, i.e.
\begin{align*}
\| \wt C_1P_n({B^*\mu} f)\|\ci{ L^2(\fD_*\oplus\fD)} \le 2\|f\|\ci{L^2(\mu;E)}. 
\end{align*}
\end{lem}

\begin{proof}
The  columns $b_k$ of $B$ are in $\cH\subset L^2(\mu;E)$, so $B^*f\mu \in L^1(\mu;\fD)$, and therefore operators $P_n^{B^*\mu}$ are bounded operators $\cH\to L^2(\fD)$. It follows from Lemma \ref{l:bound-C} that $\| \wt C_1\|_\infty \le 2$, so operator $f\mapsto \wt C_1 P_n^{B^*\mu} f$ are bounded operators $\cH\to L^2(\fD_*\oplus\fD)$ (notice that we do not claim the uniform in $n$ bounds here). 
Therefore, it is sufficient to check the  uniform boundedness on a dense set. 

Take $f=hb$ where $b\in \Ran\bB$ and  $h\in C^1(\T)$ is scalar-valued. Then for $n\in\Z$ we have by Theorem  \ref{t-repr}
\begin{align*}
\Phi^* f & - z^{n} \Phi^*(\bar\xi^{n} f ) \\
& = C_1(z)\int_\T \frac{h(\xi) - h(z)}{1-\overline \xi z} B^*b\dd\mu(\xi) - 
z^nC_1(z)\int_\T \frac{\bar\xi^{n}h(\xi) - \bar z^{n}h(z)}{1-\overline \xi z} B^*b\dd\mu(\xi) 
\\
& = C_1(z)\int_\T \frac{ 1- (\bar\xi z)^{n} }{1-\overline \xi z} B^*hb\dd\mu(\xi) 
\end{align*}
Expressing $\frac{ 1- (\bar\xi z)^{n} }{1-\overline \xi z}$ as a sum of geometric series we get that for $f=hb$, $h\in\C^1(\T)$
\begin{align*}
\Phi^* f  - z^{n} \Phi^*(\bar\xi^{n} f ) =
\left\{ \begin{array}{ll} C_1 P_{n-1}(B^*f\mu), \qquad & n\ge 1, \\ -C_1 P_n(B^*f\mu) , & n<0. \end{array}\right.
\end{align*}
By linearity the above identity holds for a dense set of linear combinations $f= \sum_k h_k b_k$, $h_k\in C^1(\T)$.  The operators $\Phi^*:\cH \to \cK_{\theta}\subset L^2(W; \fD_*\oplus \fD)$ are bounded (unitary) operators, so the desired estimate holds on the above dense set. 
\end{proof}

For a measure $\nu$ on $\T$ let $T_r\nu$ be the restriction of the Cauchy transform of $\nu$ to the circle of radius $r\ne 1$, 
\[
T_r\nu(z) = \int_\T \frac{\dd\nu(\xi)}{1- r\bar \xi z}, \qquad z\in \T. 
\]
Define operators $T_r^{B^*\mu} $ on $L^2(\mu;E)$ as 
\[
T_r^{B^*\mu} f = T_r(B^*f\mu). 
\] 

The lemma below is an immediate corollary of the above Lemma \ref{l:BdPn}. 

\begin{lem}
\label{l:BdT_r}
The operators $T_r^{B^*\mu}:\cH\subset L^2(\mu; E) \to L^2(\wt C_1^*\wt C_1;\fD_*\oplus\fD)$ are uniformly in $r$ bounded with norm at most $2$, i.e.
\begin{align*}
\|  \wt C_1 T_r^{B^*\mu} f\|\ci{L^2(\fD_*\oplus\fD)} \le 2 \|f\|\ci{L^2(\mu;E)}
\end{align*}
\end{lem}

\begin{proof}
The result follows immediately from Lemma \ref{l:BdPn}, since the operators $T_r^{B^*\mu}$ can be represented as averages of operators $P_n^{B^*\mu}$, 
\[
P_r^{B^*\mu} = \left\{
\begin{array}{ll} \displaystyle\sum_{n=0}^\infty (r^n - r^{n+1}) P_n^{B^*\mu} , \qquad & 0<r<1, \\
\displaystyle\sum_{n=1}^\infty (r^{-n} - r^{-n-1}) P_{-n}^{B^*\mu} , \qquad & r>1.  \end{array}\right.
\]
\end{proof}

Using uniform boundedness of the operators $\wt C_1 T_r^{ B^* \mu}$ (Lemma \ref{l:BdT_r}) and existence of non-tangential boundary values $T_\pm^{B^*\mu}f$ we can get the convergence of operators  $\wt C_1 T_r^{ B^* \mu}$ in the weak operator topology. 

%
%
%

\begin{prop}\label{d-TPlusMinus}
The operators $\wt C_1 T_\pm^{ B^* \mu}:\cH\subset L^2(\mu;E)\to L^2(W;\fD_*\oplus\fD)$ are bounded and
\begin{align*}
C_1 T_\pm^{B^* \mu}
=
\text{\rm w.o.t.-}\lim_{r\to 1^\mp} C_1 T_r^{ B^* \mu}.
\end{align*}
\end{prop}

\begin{proof}
We want to show that for any $f\in\cH\subset  L^2(\mu;E)$ 
\[
C_1 T_\pm^{B^*\mu} f = \text{w-}\lim_{r\to 1^{\mp}} C_1 T_r^{B^*\mu} f,  
\]
where the limit is in the weak topology of $L^2(W;\fD_*\oplus\fD)$. This is equivalent to 
\[
\wt C_1 T_\pm^{B^*\mu} f = \text{w-}\lim_{r\to 1^{\mp}} \wt C_1 T_r^{B^*\mu} f, 
\]
with the limit being in the weak topology of $L^2(\fD_*\oplus\fD)$. 

Let us prove this identity for $\wt C_1T^{B^*\mu}_+f$. Assume that for some $f\in L^2(\mu;E)$
\[
\wt C_1 T_+^{B^*\mu} f \ne \text{w-}\lim_{r\to 1^{-}} \wt C_1 T_r^{B^*\mu} f.
\]
Then for some $h\in L^2(\fD_*\oplus\fD)$
\begin{align}
\label{w-ne}
 \Bigl(\wt C_1 T_r^{B^*\mu} f, h\Bigr)\ci{L^2(\fD_*\oplus\fD)} \nrightarrow \Bigl(\wt C_1 T_+^{B^*\mu} f, h\Bigr)\ci{L^2(\fD_*\oplus\fD)} \qquad\text{as } r\to 1^-, 
\end{align}
so there exists a sequence $r_k\nearrow 1$ such that 
\[
\lim_{k\to \infty} \Bigl(\wt C_1 T_{r_k}^{B^*\mu} f, h\Bigr)\ci{L^2(\fD_*\oplus\fD)} \ne \Bigl(\wt C_1 T_+^{B^*\mu} f, h\Bigr)\ci{L^2(\fD_*\oplus\fD)}; 
\]
note that taking a subsequence we can assume without loss of generality that the limit in the left hand side exists. 

Taking a subsequence again, we can assume without loss of generality that \linebreak $\wt C_1 T_{r_k}^{B^*\mu} f \to g$ the weak topology, and \eqref{w-ne} implies that $g\ne \wt C_1 T_+^{B^*\mu} f$. 

The existence of non-tangential boundary values and the definition of $T^{B^*\mu}_+$ implies that $\wt C_1 T_{r_k}^{B^*\mu} f \to \wt C_1 T_+^{B^*\mu} f$ a.e.~on $\T$. But as \cite[Lemma 3.3]{LT09} asserts, if $f_n\to f$ a.e.~and $f_n\to g$ in the weak topology of $L^2$, then $f=g$, so we arrived at a contradiction. 

Note, that in \cite[Lemma 3.3]{LT09} everything was stated for scalar functions, but applying this scalar lemma componentwise we immediately get the same result for $L^2(\mu; E)$ with values in a separable Hilbert space.  
%
%
%
%
\end{proof}

\section{Adjoint Clark operator in Sz.-Nagy--\texorpdfstring{Foia\c s}{Foias} transcription}\label{ss-PhiStarSNF}
The main result of this section is Theorem \ref{t-repr3} below,  giving a formula for the adjoint Clark operator $\Phi^*$. 

Denote by $F$ the Cauchy transform of the matrix-valued measure $B^*B\mu$, 
\begin{align}
\label{F-02}
F(z) =\cC[B^*B\mu](z) = \int_\T \frac{1}{1-z\overline\xi} B^*(\xi) B(\xi) \dd\mu(\xi), \qquad z\in\D, 
\end{align}
and let us use the same symbol  for its non-tangential boundary values, which exist a.e. on $\T$. Using the operator $T_+^{B^*\mu}$ introduced in the previous section, we give the following formula for $\Phi^*$. 

\begin{theo}\label{t-repr3}
The adjoint Clark operator in Sz.-Nagy--Foia\c s transcription reduces to
\begin{align}
\label{Phi*-03}
\Phi^*f
=
\kf{0}{\Psi_2}  f
+
\kf{(\OID+\te\ci\Gamma \Gamma^*)D\ci{\Gamma^*}^{-1}F^{-1}}{\Delta\ci\Gamma D\ci{\Gamma}^{-1}
( \Gamma^* - \OID)}
T_+^{B^* \mu} f, \qquad f\in\cH, 
\end{align}
with $\Psi_2(z)= \wt \Psi_2(z) R(z)$, where  
\begin{align}
\label{Psi_2-alt}
\wt \Psi_2(z) & = \Delta\ci\Gamma D\ci{\Gamma}^{-1}
(\Gamma^*
+
(\OID - \Gamma^*) F(z))
\\ \notag
& =  \Delta\ci\Gamma D\ci{\Gamma}^{-1}(\bI -\Gamma^* \theta\ci\OZ(z)) F(z) \qquad \text{a.e.~on }\T,  
\end{align}
and $R$ is a measurable right inverse for the matrix-valued function $B$. 
\end{theo}

\begin{rem*}
When $d=1$, this result reduces to \cite[Equation (4.5)]{LT15}.
\end{rem*}

\begin{rem}
\label{r:Psi_well-defined}
As one should expect, the matrix-valued function $\Psi_2$ does not depend on the choice of the right inverse $R$. To prove this it is sufficient to show that $\ker B(z) \subset\ker \wt\Psi_2(z)$ a.e., which follows from the proposition below. 
\end{rem}

\begin{prop}
\label{p:Psi^*Psi}
For $\wt\Psi_2$ defined above in \eqref{wtPsi_2} and $w$ being the density of $\mu\ti{ac} $ we have 
\begin{align}
\label{Psi^*Psi}
\wt\Psi_2(\xi)^*\wt\Psi_2(\xi)= F(\xi)^* \Delta\ci\OZ(\xi)^2 F(\xi) &=  B(\xi)^*B(\xi) w(\xi)\qquad && \mu\ti{ac}\text{-a.e.,}
\intertext{and so}
\label{Psi^*Psi-01}
\Psi_2(\xi)^*\Psi_2(\xi) & = w(\xi) \bI\ci{E(\xi)} && \mu\ti{ac}\text{-a.e.}
\end{align}
\end{prop}


\begin{proof}
Since $\Psi_2 =\wt \Psi_2 R$, \eqref{Psi^*Psi-01} follows immediately from \eqref{Psi^*Psi}. 

To prove \eqref{Psi^*Psi},  
consider first the case $\Gamma = \OZ$. In this case $\wt\Psi = \Delta\ci\OZ F$, so
\begin{align}
\notag
\wt \Psi^*_2 \wt\Psi_2 & = F^* \Delta\ci\OZ^2 F  = (\bI - \theta\ci\OZ^*)^{-1} \Delta\ci\OZ^2 (\bI - \theta\ci\OZ)^{-1} && 
\\
\label{Psi^*Psi-02}
& = B^*Bw . && \text{by \eqref{e-NOT} }
\end{align}

Consider now the case of general $\Gamma$. We get 
\begin{align*}
\wt \Psi_2^* \wt\Psi_2 & = F^*(\bI - \theta\ci\OZ^*\Gamma) D\ci{\Gamma}^{-1}\Delta\ci\Gamma^2 D\ci{\Gamma}^{-1}(\bI -\Gamma^* \theta\ci\OZ) F && 
\\
& = F^* \Delta\ci\OZ^2 F  && \text{by Theorem \ref{t-Delta}} 
\\
& = B^*Bw  && \text{by \eqref{Psi^*Psi-02}. }
\end{align*}
\end{proof}

%

\subsection{A preliminary formula}
\label{s:Phi^*-prelim}

We start proving Theorem \ref{t-repr3} by first proving this preliminary result, that holds for any transcription of the model. Below the matrix-valued functions $C_*$ and $C$ are from Theorem \ref{t-repr}, 
and $C_1(z):= C_*(z)-zC(z)$. 

\begin{prop}\label{p:repr2}
The adjoint  Clark operator represented for $f\in \cH\subset L^2(\mu;E)$ by
\begin{align}
\label{Phi^*NF-02}
(\Phi^* f)(z) = 
C_1(z)
(T_\pm^{B^* \mu} f)(z)
+\Psi_\pm(z) f(z)
,\quad z\in \T,
\end{align}
where the matrix-functions $\Psi_\pm$, $\Psi_\pm(z): E(z) \to \C^{2d}=\fD_*\oplus\fD$ are defined via the identities 
\begin{align}
\label{Psi}
\Psi_\pm(z)b(z) :=C_*(z){\bB}^*b-C_1(z)(T_\pm^{B^* \mu} b)(z), \qquad b\in \Ran \bB;
\end{align}
here two choices of sign (the same sign for all terms) gives two different representation formulas. 
\end{prop}

%

\begin{rem*}
When $d=1$ and $b\equiv 1$ this alternative representation formula reduces to a formula that occurs in the proof of \cite[Theorem 4.7]{LT15}.
\end{rem*}

\begin{rem*}
It is clear that relations \eqref{Psi} with $b=b_k$, $k=1,2,\ldots, d$, completely defines 
the matrix-valued function $\Psi$.  However, it is not immediately clear why such function $\Psi$ exists; the existence of $\Psi$ will be shown in the proof. 

Recalling the definition \eqref{F-02} of the function $F$, 
we can see that $\Psi (z) b_k(z)$ can be given as the (non-tangential) boundary values of the vector-valued  function 
\begin{align}
\label{Psi-01}
C_*(z) e_k - C_1(z) F(z) e_k, \qquad z\in\D,  
\end{align}
where $e_1, e_2, \ldots, e_d$ is the standard orthonormal basis in $\C^d$. 
\end{rem*}
%

%
%

%

\begin{proof}[Proof of Proposition \ref{p:repr2}]
Let us first show the result for functions of the form $f=hb\in L^2(\mu; E)$, where $b\in \Ran\bB$ and $h$ is a scalar function. We want to show that  
\begin{align}\label{t-reprL2}
(\Phi^* hb)(z)
=
C_1(z)
(T_\pm^{B^* \mu} hb )(z)
+h(z)\psi_b^\pm(z)
,\quad z\in \T,
\end{align}
where
\[
\psi_b^\pm(z)
:=
C_*(z){\bB}^*b
-C_1(z)
(T_\pm^{B^* \mu} b)(z).
 \]

First note that \eqref{e-repr} implies that for $b\in \Ran \bB$ 
\begin{align*}
\Phi^* b (z) = C_*(z) \bB^* b.
\end{align*}
Observe that for (scalar) $h\in C^1$ we have uniform on $z\in\T$ convergence as $r\to 1^\mp$:
\begin{align}
\label{uniform_conv-01}
\int\ci{\T}
 \frac{h(\xi)-h(z)}{1-rz\bar\xi}B^*(\xi)b(\xi)
d\mu(\xi)
&\rightrightarrows 
 \int\ci{\T}
 \frac{h(\xi)-h(z)}{1-z\bar\xi}B^*(\xi)b(\xi)
d\mu(\xi).
\end{align}
Multiplying both sides by $C_1(z)$ we get in
the left hand side exactly $C_1(z) (T_r^{B^* \mu} hb)(z) - h(z) C_1(z) (T_r^{B^* \mu} b)(z)$, 
and in the right hand side  the part with the integral  in the representation \eqref{e-repr}. 

Recall that the model space $\cK\ci{\theta_\Gamma}$ is a subspace of a weighted space $L^2(W, \fD_*\oplus\fD)$. 
Uniform convergence in \eqref{uniform_conv-01} implies the convergence in $L^2(\fD_*\oplus\fD)$, and by Lemma \ref{l:bound-C} the multiplication by $C_*$ and $C_1$ are bounded operators $L^2(\fD)\to L^2(W;\fD_*\oplus\fD)$.  Thus (because $h$ is bounded) 
\[
h C_*\bB^* b + C_1 T_r^{B^* \mu} hb - h C_1T_r^{B^* \mu} b \to \Phi^* hb
\]
as $r\to 1^\mp$ in the norm of $L^2(W;\fD_*\oplus\fD)$. By Proposition \ref{d-TPlusMinus} the operators $C_1 T_r^{B^*\mu} \to C_1 T_\pm^{B^*\mu}$ in weak operator topology as $r\to 1^\mp$, so
\begin{align}
\label{form-repr-01}
\Phi^* hb = C_1 T_\pm^{B^* \mu} hb + h C_*\bB^* b - h C_1T_\pm^{B^* \mu} b ,
\end{align}
which immediately implies \eqref{t-reprL2}. Thus, \eqref{t-reprL2} is proved for $h\in C^1(\T)$.

To get \eqref{form-repr-01}, and so \eqref{t-reprL2} for for general $h$ such that $hb\in L^2(\mu;E)$ (recall that $b\in\Ran \bB$) we use the standard approximation argument: the operators $\Phi^*, C_1 T_\pm^{B^*\mu}:\cH\to 
 L^2(W;\fD_*\oplus \fD)$ are bounded, and therefore for a fixed $b\in\Ran\bB$ the operators $hb \mapsto h\psi_b^\pm$ (which are defined initially on a submanifold of $\cH$ consisting of functions of form $hb$, $h\in C^1(\T)$) are bounded (as a difference of two bounded operators). Approximating  in $L^2(\mu;E)$ the function $hb$ by functions $h_n b$, $h_n\in C^1(\T)$ we get  \eqref{form-repr-01} and \eqref{t-reprL2} for general $h$. 

Let us now proof existence of $\Psi$. 
Consider the (bounded) linear operator  $\Phi^* - C_1 T^{B^*\mu}$. We know that for $f=hb\in L^2(\mu;E)$ with $b\in\Ran \bB$ and scalar $h$ 
\[
(\Phi^* - C_1 T_\pm^{B^*\mu})hb = h\psi_b^\pm, 
\]
so on functions $f=hb$ the operators $\Phi^* - C_1 T_\pm^{B^*\mu}$ intertwine the multiplication operators $M_\xi$ and $M_z$. Since linear combinations of functions $h_kb_k$ are dense in $\cH$, we conclude that the operators $\Phi^* - C_1 T_\pm^{B^*\mu}$ intertwine $M_\xi$ and $M_z$ on all $\cH$, and so these operators are  the multiplications by some matrix functions $\Psi_\pm$. 

Using  \eqref{form-repr-01} with $h=1$ we can see that 
\[
\Psi_\pm b = \Phi^* b -C_1 T^{B^*\mu}_\pm b = \bC_* B^* b - C_1 T^{B^*\mu}_\pm b, 
\]
so $\Psi_\pm$ are defined exactly as stated in the proposition. 
%
%
%
\end{proof}

\subsection{Some calculations}\label{ss-repr3}

Let us start with writing more detailed formulas for the matrix functions $C_*$ and $C_1$ from  Proposition \ref{p:repr2}. 

\begin{lem}\label{l-C1simple}
We have
\begin{align*}
C_*(z)
&=
\kf{ \OID+ {\te}\ci\Gamma (z)\Gamma^\ast}{ \Delta\ci\Gamma (z)\Gamma^\ast} 
D\ci{\Gamma^*}^{-1},
\qquad
C_1(z)
&=
\kf{\bI}{\OZ}D\ci{\Gamma^*}^{-1}(\OID - \Gamma)+\kf{\te\ci\Gamma(z)}{\Delta\ci\Gamma(z)}D\ci{\Gamma}^{-1}(\Gamma^*-\OID).
\end{align*}
\end{lem}

\begin{proof}
The formula for $C_*(z)$ is just \eqref{C*_N-F} and the identity $\theta\ci\Gamma(0)=-\Gamma$. 
Similarly, equation \eqref{C_N-F} gives us
\[
C(z)
=
\kf{z^{-1}(\theta\ci\Gamma (z)+\Gamma)}{z^{-1}\Delta\ci\Gamma(z)}D\ci{\Gamma}^{-1}.
\]
Substituting these expressions into $C_1(z) = C_*(z) - z C(z)$ and applying the commutation relations from Lemma \ref{l:G-D_G} we see
\begin{align*}
C_1(z)
&
=
\kf{D\ci{\Gamma^*}^{-1}+ \theta\ci\Gamma \Gamma^* D\ci{\Gamma^*}^{-1}-\theta\ci\Gamma  D\ci{\Gamma}^{-1}-\Gamma D\ci{\Gamma}^{-1}}{\Delta\ci\Gamma \Gamma^*D\ci{\Gamma^*}^{-1}-\Delta\ci\Gamma D\ci{\Gamma}^{-1}}
\\
& =
\kf{D\ci{\Gamma^*}^{-1}+ \theta\ci\Gamma  D\ci{\Gamma}^{-1}\Gamma^*-\theta\ci\Gamma  D\ci{\Gamma}^{-1}- D\ci{\Gamma^*}^{-1}\Gamma}{\Delta\ci\Gamma D\ci{\Gamma}^{-1}\Gamma^*-\Delta\ci\Gamma D\ci{\Gamma}^{-1}}
\\&
=
\kf{D\ci{\Gamma^*}^{-1}(\OID-\Gamma)+ \theta\ci\Gamma  D\ci{\Gamma}^{-1}(\Gamma^*-\OID)}{\Delta\ci\Gamma D\ci{\Gamma}^{-1}(\Gamma^*-\OID)}
\\ & =
\kf{\bI}{\OZ}D\ci{\Gamma^*}^{-1}(\OID-\Gamma)
+
\kf{\theta\ci\Gamma }{\Delta\ci\Gamma} D\ci{\Gamma}^{-1}(\Gamma^*-\OID)
,
\end{align*}
and the second statement in the lemma is verified.
\end{proof}

Recall that $F(z)$, $z\in\D$ is the matrix-valued Cauchy transform of the measure  $B^*B\mu$, see  \eqref{F-02}, and that for $z\in\T$ the symbol $F(z)$ denotes the non-tangential boundary values of $F$. We need the following simple relations between $F$ and $\theta\ci\OZ$. 

\begin{lem}
\label{l:theta-F}
For all $z\in\D$ and a.e.~on $\T$
\begin{align*}
F(z)= ( \bI - \theta\ci\OZ(z) )^{-1};
\end{align*}
note that for all $z\in\D$ the matrix $\theta\ci\OZ(z)$ is a strict contraction, so $\bI -\theta\ci\OZ(z) $ is invertible.
\end{lem}
\begin{proof}
Recall that the function $F_1$ was defined by $F_1(z)=\cC_1[B^*B\mu](z)$. Since $F(z)=\bI+ F_1(z)$, we get from \eqref{theta_0-01} that 
\[
\theta\ci\OZ (z) = F_1(z) (\bI + F_1(z))^{-1} = \bigl( F(z) - \bI \bigr) F(z)^{-1}. 
\]
Solving for $F$ we get the conclusion of the lemma. 
\end{proof}

\subsection{Proof of Theorem \ref{t-repr3}}
Let us first prove the second identity in \eqref{Psi_2-alt}. Using the identity $F=(\bI - \theta\ci\OZ)^{-1}$ we compute
\begin{align*}
\Gamma^* + (\bI -\Gamma^* ) F = (\Gamma^* (\bI -\theta\ci\OZ ) + \bI -\Gamma^*) F = (\bI - \Gamma^* \theta\ci\OZ)F,  
\end{align*}
which is exactly what we need. 

Let us now prove that $\Psi$ from Proposition \ref{p:repr2} if given by $\Psi =\kf{\OZ}{\Psi_2}$ with $\Psi_2$ defined above in Theorem \ref{t-repr3}. Since $R(z) b_k(z)= e_k$, it is sufficient to show that $\Psi =\kf{\OZ}{\Psi_2}$ and that 

\begin{align}
\label{Psi_2}
\Psi_2(z) b_k(z) =\Delta\ci\Gamma D\ci{\Gamma}^{-1}
(\Gamma^*
+
(\OID - \Gamma^*)
F(z))e_k, \qquad k=1, 2, \ldots, d. 
\end{align}
%
%

Using the formulas for $C_*$ and $C_1$ provided in Lemma \ref{l-C1simple} we get from \eqref{Psi-01} 
\begin{align*}
\Psi(z) b_k(z) & = C_*(z) e_k - C_1(z) F(z) e_k 
\\
&= 
\kf{
(\OID+{\te}\ci\Gamma \Gamma^\ast)D\ci{\Gamma^*}^{-1}
-
[D\ci{\Gamma^*}^{-1}(\OID- \Gamma)+ {\te}\ci\Gamma D\ci{\Gamma}^{-1} (\Gamma^\ast-\OID)]
F
}
{
\Delta\ci\Gamma \Gamma^* D\ci{\Gamma^*}^{-1}
-\Delta\ci\Gamma D\ci{\Gamma}^{-1}
( \Gamma^*-\OID)
F
}
e_k
.
\end{align*}

Note that it is clear from the representation \eqref{Phi^*NF-02} that the top entry of $\Psi$ should disappear, i.e.~that 
\begin{align}
\label{Psi_1=0}
(\OID+{\te}\ci\Gamma \Gamma^\ast)D\ci{\Gamma^*}^{-1} = [D\ci{\Gamma^*}^{-1}(\OID- \Gamma)+ {\te}\ci\Gamma D\ci{\Gamma}^{-1} (\Gamma^\ast-\OID)] F.
\end{align}
Indeed, by the definition of $\cK\ci\theta$ in the Sz.-Nagy--Foia\c{s} transcription the top entry of $\Phi^* f$ belongs to $H^2(\fD_*)$. One can see from Lemma  \ref{l-C1simple}, for example, that the top entry of $C_1$ belongs to matrix-valued $H^\infty$, so the top entry of $C_1 T_+^{B^*\mu} f$ is also in $H^2(\fD_*)$. Therefore the top entry of $\Psi f$ must be in $H^2(\fD_*)$ for all $f$. But that is impossible, because $f$ can be any function in $L^2(\mu;E)$. 

For a reader that is not comfortable with such ``soft'' reasoning, we present a ``hard'' computational proof of \eqref{Psi_1=0}. This computation also helps to assure the reader that the previous computations were correct. 

To do the computation, consider the term in the square brackets in the right hand side of \eqref{Psi_1=0}. 
Using the commutation relations from Lemma \ref{l:G-D_G} in the second equality, we get
\begin{align*}
D\ci{\Gamma^*}^{-1}(\OID- \Gamma) + {\te}\ci\Gamma D\ci{\Gamma}^{-1}( \Gamma^\ast-\OID)
&=
D\ci{\Gamma^*}^{-1}+ \te D\ci{\Gamma}^{-1}\Gamma^*-\te D\ci{\Gamma}^{-1}- D\ci{\Gamma^*}^{-1}\Gamma
\\
&=
D\ci{\Gamma^*}^{-1}+ \te\Gamma^* D\ci{\Gamma^*}^{-1}-\te D\ci{\Gamma}^{-1}-\Gamma D\ci{\Gamma}^{-1}
\\
&=
(\OID+\te\ci\Gamma \Gamma^*)D\ci{\Gamma^*}^{-1}
\{\bI - D\ci{\Gamma^*}(\OID+\te\ci\Gamma \Gamma^*)^{-1}(\te\ci\Gamma+\Gamma)D\ci{\Gamma}^{-1}\}
\\
& = (\OID+\te\ci\Gamma \Gamma^*)D\ci{\Gamma^*}^{-1} \{\bI-\theta\ci\OZ\};
\end{align*}
the last equality holds by Theorem \ref{t:LFT-02}. 

By Lemma \ref{l:theta-F} we have $\OID - \te\ci\OZ = F^{-1}$, 
so we have for the term in the square brackets 
\[
[D\ci{\Gamma^*}^{-1}(\OID- \Gamma)+ {\te}\ci\Gamma D\ci{\Gamma}^{-1} (\Gamma^\ast-\OID)] = 
(\OID+\te\ci\Gamma \Gamma^*)D\ci{\Gamma^*}^{-1} F^{-1},  
\]
which proves \eqref{Psi_1=0}.

To deal with the bottom entry of $\Psi$ we use the commutation relations from Lemma \ref{l:G-D_G},
\begin{align*}
\Delta\ci\Gamma \Gamma^* D\ci{\Gamma^*}^{-1} -\Delta\ci\Gamma D\ci{\Gamma}^{-1} ( \Gamma^*-\OID) F & =
\Delta\ci\Gamma D\ci{\Gamma}^{-1} \Gamma^* - \Delta\ci\Gamma D\ci{\Gamma}^{-1}  \Gamma^* F + \Delta\ci\Gamma D\ci{\Gamma}^{-1}  F
\\
& = \Delta\ci\Gamma D\ci{\Gamma}^{-1} \left( \Gamma^* + (\bI - \Gamma^*) F \right) , 
\end{align*}
which gives the desired formula \eqref{Psi_2} for $\Psi_2$.

Finally,  let us deal with the second term in the right had side of \eqref{Phi*-03}. We know from Proposition \ref{p:repr2} that the term in front of $T^{B^*\mu}_+ f$ is given by $C_1$. From Lemma \ref{l-C1simple} we get 
\begin{align*}
C_1 = 
\kf{
 D\ci{\Gamma^*}^{-1}(\OID- \Gamma)+ {\te}\ci\Gamma D\ci{\Gamma}^{-1} (\Gamma^\ast-\OID)
}
{
\Delta\ci\Gamma D\ci\Gamma^{-1} (\Gamma^*-\bI)
}.
\end{align*}
But the top entry of $C_1$ here is the expression in brackets in the right hand side of \eqref{Psi_1=0}, so it is equal to $(\OID+{\te}\ci\Gamma \Gamma^\ast)D\ci{\Gamma^*}^{-1}F^{-1}$. Therefore 
\[
C_1 = 
\kf{
 (\OID+{\te}\ci\Gamma \Gamma^\ast)D\ci{\Gamma^*}^{-1}F^{-1}
}
{
\Delta\ci\Gamma D\ci\Gamma^{-1} (\Gamma^*-\bI)
}, 
\]
which is exactly what we have in \eqref{Phi*-03}. 
\  \hfill \qed

\subsection{Representation of \texorpdfstring{$\Phi^*$}{Phi*} using matrix-valued measures}
The above Theorem \ref{t-repr3} is more transparent if we represent the direct integral $\cH$ as the weighted $L^2$ space with a matrix-valued measure. 

Namely, consider the weighted space $L^2(B^*B\mu)$ 
\[
\|f\|\ci{L^2(B^*B\mu)}^2 = \int_\T \bigl(B(\xi)^*B(\xi) f(\xi), f(\xi) \bigr)\ci{\C^d} \dd\mu(\xi) =\int_\T \| B(\xi) f(\xi)\|^2\ci{\C^d} \dd\mu(\xi) 
\]
(of course one needs to take the quotient space over the set of function with norm $0$). 

Then for all scalar functions $\f_k$ we have 
\[
\biggl\|\sum_{k=1}^d \f_k e_k \biggr\|_{L^2(B^*B\mu)} = \biggl\|\sum_{k=1}^d \f_k b_k \biggr\|_{L^2} ;
\]
recall that $e_1, e_2, \ldots, e_d$ is the standard basis in $\C^d$ and $b_k(\xi) = B(\xi)e_k$.
Then the map $\cU$
\[
\cU \biggl( \sum_{k=1}^d \f_k e_k \biggr) =  \sum_{k=1}^d \f_k b_k, \qquad \text{or, equivalently }\quad \cU f = Bf, 
\]
defines a unitary operator from $L^2(B^*B\mu)$ to $\cH$. 

The inverse operator $\cU^*$ is  given by $\cU^* f (\xi) = R(\xi) f(\xi)$, where, recall,  $R$ is a measurable pointwise right inverse of $B$, $B(\xi) R(\xi) = \bI\ci{E(\xi)}$ $\mu$-a.e.

We denote by $\wt\Phi := \cU^* \Phi$, so $\wt\Phi^* = \Phi^* \cU$, and by $T_+^{B^*B\mu} f$ the non-tangential boundary values of the Cauchy integral $\cC[B^*Bf\mu](z)$, $z\in\D$. Substituting $f=Bg$ into \eqref{Phi*-03} we can restate Theorem \ref{t-repr3}  as follows.


\begin{theo}
\label{t:repr04}
The adjoint Clark operator $\wt\Phi^*:L^2(B^*B\mu) \to \cK_\theta$ in Sz.-Nagy--Foia\c s transcription is given by 
\begin{align}
\label{Phi*-04}
\wt\Phi^*g
=
\kf{0}{\wt\Psi_2}  g
+
\kf{(\OID+\te\ci\Gamma \Gamma^*)D\ci{\Gamma^*}^{-1}F^{-1}}{\Delta\ci\Gamma D\ci{\Gamma}^{-1}
( \Gamma^* - \OID)}
T_+^{ B^*B  \mu} g, \qquad g\in L^2(B^*B\mu), 
\end{align}
 where the matrix-valued function $\wt\Psi_2(z)$ is defined as 
\begin{align}
\label{wtPsi_2}
\wt\Psi_2(z)  =\Delta\ci\Gamma D\ci{\Gamma}^{-1}
(\Gamma^*
+
(\OID - \Gamma^*) F(z))  . 
\end{align}
\end{theo}


%

%

\subsection{A generalization of the normalized Cauchy transform}
\label{ss:GCT}
Consider the case when the unitary operator $U$ has purely singular spectrum. By virtue of Corollary \ref{c-TFAE}, the second component of the Sz.-Nagy--Foia\c s model space collapses, i.e.~$\cK_{\te\ci\Gamma} = H^2(\C^d) \ominus \te\ci\Gamma H^2(\C^d)$ for all strict contractions $\Gamma$.

The  representation formula \eqref{Phi*-03} then reduces to a generalization of the well-studied normalized Cauchy transform.


\begin{cor}\label{c-KapustinPolt}
If $\te = \theta\ci\OZ$ is inner, then
\[
(\Phi^* f)(z)
=
(\OID-\theta(z)) (T_+^{B^* \mu}f)(z)
=
(F(z))^{-1} (T_+^{B^* \mu}f)(z)
\]
for $z\in \D, f\in L^2(\mu;E)$.
\end{cor}

The first equation was also obtained in \cite[Theorem 1]{KP06}.

Here we used $\Gamma = \OZ$ only for simplicity. With the linear fractional relation \eqref{t:LFT-02}, it is not hard to write the result in terms of $\te\ci\Gamma$ for any strict contraction $\Gamma$.

\begin{proof}
Theorem \ref{t-repr3} for inner $\te$ and $\Gamma=\OZ$ immediately reduces to the first statement.
%
%
%

The equality of the second expression follows immediately from Lemma \ref{l:theta-F}.
\end{proof}

%

\section{The Clark operator}
\label{s:directClark}

Let $f\in\cH\subset L^2(\mu; E)$ and let
\begin{align}
\label{Phi^*f=h}
\Phi^* f = h = \left(\begin{array}{c} h_1\\ h_2\end{array}\right) \in\cK\ci\theta .
\end{align}
From the representation \eqref{Phi*-04} we get, subtracting from the second component the first component multiplied by an appropriate matrix-valued function, that 
\begin{align*}
\Psi_2 f = h_2 -\Delta\ci\Gamma D\ci\Gamma^{-1}(\Gamma^*-\bI) F D\ci{\Gamma^*} (\bI+\theta\ci\Gamma \Gamma^*)^{-1} h_1.
\end{align*}
Right multiplying this identity by $\Psi_2^*$, and using Proposition \ref{p:Psi^*Psi} and formulas for $\Psi_2$, $\wt\Psi_2$ from Theorem \ref{t-repr3},   we get an expression for the density of the absolutely continuous part of $\mu\ti{ac}$. Namely, we find that a.e.~(with respect to Lebesgue measure on $\T$)
\begin{align}
\label{g_{ac}-01}
w  f & = R^*F^*(\bI - \theta\ci\OZ^*\Gamma) D\ci{\Gamma}^{-1}\Delta\ci\Gamma h_2 
\\ \notag 
& \qquad -
R^*F^*(\bI - \theta\ci\OZ^*\Gamma) D\ci{\Gamma}^{-1}\Delta\ci\Gamma^2 D\ci\Gamma^{-1}(\Gamma^*-\bI) F D\ci{\Gamma^*} (\bI+\theta\ci\Gamma \Gamma^*)^{-1} h_1 
\\ \notag 
& = R^*F^*(\bI - \theta\ci\OZ^*\Gamma) D\ci{\Gamma}^{-1}\Delta\ci\Gamma h_2 
\\ \notag  
& \qquad - R^*F^* \Delta\ci\OZ^2 (\bI - \Gamma^*\theta\ci\OZ)^{-1} (\Gamma^*-\bI) F D\ci{\Gamma^*} (\bI+\theta\ci\Gamma \Gamma^*)^{-1} h_1 .
\end{align}
In the case $\Gamma=\OZ$ the above equation simplifies:
\begin{align}
\label{g_{ac}-02}
w f & = R^*F^*\Delta\ci\OZ h_2 + R^*F^* \Delta\ci\OZ^2 F h_1 
\\ \notag 
& = R^*F\Delta\ci\OZ h_2 + wB h_1;
\end{align}
in the second equality we use \eqref{Psi^*Psi}. 

The above formulas \eqref{g_{ac}-01}, \eqref{g_{ac}-02} determine the absolutely continuous part of $f$. 

The singular part of $f$ was in essence computed in \cite{KP06}.  Formally it was computed there  only for inner functions $\theta$, but using the ideas and results from \cite{KP06} it is easy to get the general case from our Theorem \ref{t-repr3}. 


For the convenience of the reader, we give a self-contained presentation.

\begin{lem}
\label{l:polt}
Let $f\in L^2(\T, \mu; \C^d)$. Then $\mu\ti s$-a.e.~the nontagential boundary values of $\cC [f\mu](z)/\cC[\mu](z)$, $z\in\D$ exist and equal $f(\xi)$, $\xi\in\T$.  
\end{lem}

This lemma was proved in  \cite{KP06} even for a more general case of $f\in L^2(\mu;E)$, where $E$ is a separable Hilbert space. Note that  our case $E=\C^d$ follows trivially by applying  the corresponding scalar result ($E=\C$) proved in \cite{NONTAN} to entries of the vector $f$.

Applying the above Lemma to the representation giving by the first coordinate of \eqref{Phi*-03} from Theorem \ref{t-repr3} we get that for $f$ and $h$ related by \eqref{Phi^*f=h} we have 
\begin{align*}
B^*f  = \frac1{\cC[\mu]} F D\ci{\Gamma^*} (\bI+\theta\ci \Gamma \Gamma^*)^{-1} h_1 \qquad \mu\ti s\text{-a.e.}
\end{align*}
Left multiplying this identity by $R^*$ we get that 
\begin{align}
\label{f_sing}
\Phi h = f = \frac1{\cC[\mu]} R^* F D\ci{\Gamma^*} (\bI+\theta\ci \Gamma \Gamma^*)^{-1} h_1 \qquad \mu\ti s\text{-a.e.}
\end{align}

Summarizing, we get the following theorem, describing the direct Clark operator $\Phi$.
\begin{theo}
\label{t:direct Clark}
If $\Phi^* f = h$ as in \eqref{Phi^*f=h}, so $f=\Phi h$, then the absolutely continuous part of $f$ is given by \eqref{g_{ac}-01} and the singular part of $f$ is given by \eqref{f_sing}.
\end{theo}

\providecommand{\bysame}{\leavevmode\hbox to3em{\hrulefill}\thinspace}
\providecommand{\MR}{\relax\ifhmode\unskip\space\fi MR }
\providecommand{\MRhref}[2]{%
  \href{http://www.ams.org/mathscinet-getitem?mr=#1}{#2}
}
\providecommand{\href}[2]{#2}

\end{document}